\documentclass[a4paper]{amsart}

\usepackage{enumerate, amsmath, amsfonts, amssymb, amsthm, wasysym, graphics, graphicx, xcolor, url, hyperref, hypcap, a4wide, stmaryrd, shuffle, xargs, multicol, overpic, pdflscape, multirow, hvfloat, minibox, accents, etoolbox, dsfont}
\hypersetup{colorlinks=true, citecolor=darkblue, linkcolor=darkblue, urlcolor=darkblue}
\usepackage[all]{xy}
\usepackage[bottom]{footmisc}
\usepackage{tikz}\usetikzlibrary{trees,snakes,shapes,arrows,matrix,calc}
\graphicspath{{figures/}}
\makeatletter
\def\input@path{{figures/}}
\makeatother

\newtheorem{theorem}{Theorem}
\newtheorem{corollary}[theorem]{Corollary}
\newtheorem{proposition}[theorem]{Proposition}
\newtheorem{lemma}[theorem]{Lemma}
\newtheorem{definition}[theorem]{Definition}

\theoremstyle{definition}
\newtheorem{example}[theorem]{Example}

\newcommand{\N}{\mathbb{N}} % naturals
\newcommand{\fS}{\mathfrak{S}} % symmetric group
\renewcommand{\b}[1]{\mathbf{#1}} % bold letters
\renewcommand{\c}[1]{\mathcal{#1}} % calligraphic letters
\newcommand{\f}[1]{\mathfrak{#1}} % frak letters

\newcommand{\set}[2]{\left\{ #1 \;\middle|\; #2 \right\}} % set notation
\newcommand{\bigset}[2]{\big\{ #1 \;|\; #2 \big\}} % big set notation
\newcommand{\ssm}{\smallsetminus} % small set minus
\newcommand{\eqdef}{\mbox{\,\raisebox{0.2ex}{\scriptsize\ensuremath{\mathrm:}}\ensuremath{=}\,}} % :=
\newcommand{\Id}{\mathrm{Id}} % identity
\DeclareMathOperator{\inv}{inv} % inversions

\newcommandx{\graphG}[1][1=G]{\mathrm{#1}} % graph
\newcommandx{\tree}[1][1=T]{\mathrm{#1}} % tree
\newcommand{\decoration}{\delta} % decoration of a Cambrian tree
\newcommand{\linearExtensions}{\mathcal{L}} % linear extensions
\newcommand{\eP}{\mathfrak{P}} % decorated permutations
\newcommand{\eD}{\mathfrak{D}} % decorated diagrams
\newcommand{\juxta}[2]{#1#2} % concatenation
\newcommand{\concatf}{\mathsf{concat}} % concatenation
\newcommand{\selectf}{\mathsf{select}} % selection
\newcommand{\stdf}{\mathsf{std}} % standardization
\newcommand{\concat}[2]{\concatf(#1,#2)} % concatenation
\newcommand{\select}[2]{\selectf(#1,#2)} % selection
\newcommand{\std}[2]{\stdf(#1,#2)} % standardization
\newcommand{\stdpos}[2]{\mathsf{stdp}(#1,#2)} % standardization positions
\newcommand{\stdval}[2]{\mathsf{stdv}(#1,#2)} % standardization values
\newcommand{\arcs}{\Psi} % arc ideal
\newcommand{\north}{\b{n}} % north
\newcommand{\south}{\b{s}} % south
\newcommand{\east}{\b{e}} % east
\newcommand{\west}{\b{w}} % west

\newcommandx{\Perm}[1][1=n]{\mathds{P}\mathrm{erm}(#1)} % permutahedron
\newcommandx{\Asso}[1][1=n]{\mathds{A}\mathrm{sso}(#1)} % associahedron
\newcommandx{\Para}[1][1=n]{\mathds{P}\mathrm{ara}(#1)} % parallelepiped
\newcommandx{\Zono}[1][1=\decoration]{\mathds{Z}\mathrm{ono}(#1)} % zonotope
\newcommandx{\Permutreehedron}[1][1=\decoration]{\mathds{PT}(#1)} % permutreehedron
\newcommand{\fan}{\mathcal{F}} % fan
\newcommandx{\Fan}[1][1=\decoration]{\fan(#1)} % permutree fan
\renewcommand{\MR}{\mathsf{MR}} % Malvenuto-Reutenauer
\newcommand{\LR}{\mathsf{LR}} % Loday-Ronco
\newcommand{\swap}{\mathrm{swap}} % swap

\newcommand{\product}{\cdot} % product
\newcommand{\coproduct}{\triangle} % coproduct
\newcommand{\shiftedShuffle}{\,\bar\shuffle\,} % shifted shuffle
\newcommand{\convolution}{\star} % convolution
\newcommand{\F}{\mathbb{F}} % F-basis of FQSym
\newcommand{\PPT}{\mathbb{P}} % P-basis of PT
\newcommandx{\surjection}[2][1=\decoration, 2=\decoration']{\Psi_{#1}^{#2}} % surjection from #1-permutrees to #2-permutrees
\newcommandx{\surjectionSchroder}[2][1=\decoration, 2=\decoration']{{\Psi^\star}_{#1}^{#2}} % surjection from #1-permutrees to #2-permutrees
\newcommand{\meet}{\wedge} % meet in lattice
\newcommand{\join}{\vee} % join in lattice
\newcommand{\bigMeet}{\bigwedge} % meet
\newcommand{\bigJoin}{\bigvee} % join
\newcommand{\projDown}{\pi^{\equiv}_{\!\downarrow\!}} % down projection lattice congruence
\newcommand{\projUp}{\pi_{\equiv}^{\!\uparrow\!}} % up projection lattice congruence

\newcommand{\fref}[1]{Figure~\ref{#1}} % reference figures
\newcommand{\ie}{\textit{i.e.}~} % id est
\newcommand{\viceversa}{\textit{vice versa}} % exempli gratia
\definecolor{darkblue}{rgb}{0,0,0.7} % darkblue color
\newcommand{\darkblue}{\color{darkblue}} % darkblue command
\newcommand{\red}{\color{red}} % red command
\newcommand{\blue}{\color{blue}} % blue command
\newcommand{\defn}[1]{\emph{\darkblue #1}} % emphasis of a definition
\usepackage{todonotes}

\makeatletter
\def\l@section{\@tocline{1}{3pt}{0pc}{}{}}
\makeatother
\let\oldtocpart=\tocpart
\renewcommand{\tocpart}[2]{\hspace{0em}\bf\large\oldtocpart{#1}{#2}}
\let\oldtocsection=\tocsection
\renewcommand{\tocsection}[2]{\hspace{0em}\bf\oldtocsection{#1}{#2}}

\title{Hopf algebras on decorated noncrossing arc diagrams}

\author{Vincent Pilaud}
\address{CNRS \& LIX, \'Ecole Polytechnique, Palaiseau}
\email{vincent.pilaud@lix.polytechnique.fr}
\urladdr{http://www.lix.polytechnique.fr/~pilaud/}

\thanks{Partially supported by the French ANR grants SC3A~(15\,CE40\,0004\,01) and CAPPS~(17\,CE40\,0018).}

\begin{document}

\begin{abstract}
Noncrossing arc diagrams are combinatorial models for the equivalence classes of the lattice congruences of the weak order on permutations.
In this paper, we provide a general method to endow these objects with Hopf algebra structures.
Specific instances of this method produce relevant Hopf algebras that appeared earlier in the literature.
\end{abstract}

\maketitle

%%%%%%%%%%%%%%%%%%%%%%%%%%%%%%%%%%%%%%

\section{Introduction}

Combinatorial Hopf algebras are combinatorial vector spaces endowed with a product (that combines combinatorial objects) and a coproduct (that decomposes combinatorial objects), subject to a strong compatibility relation.
This paper is motivated by two particularly relevant combinatorial Hopf algebras: C.~Malvenuto and C.~Reutenauer's Hopf algebra~$\MR$ on permutations~\cite{MalvenutoReutenauer} and J.-L.~Loday and M.~Ronco's Hopf algebra~$\LR$ on binary trees~\cite{LodayRonco}.
Remarkably~\cite{HivertNovelliThibon-algebraBinarySearchTrees}, $\LR$ embeds as a Hopf subalgebra of~$\MR$ by sending each binary tree~$\tree$ to the sum of the permutations in a certain class~$\linearExtensions(\tree)$.
More precisely, the permutations in~$\linearExtensions(\tree)$ are the linear extensions of~$\tree$ (seen as the Hasse diagram of a poset oriented towards its root), or equivalently the permutations whose insertion in a binary search tree gives~$\tree$.
The resulting classes are equivalence classes of the sylvester congruence~\cite{HivertNovelliThibon-algebraBinarySearchTrees} on permutations, defined as the transitive closure of the rewriting rule~$UacVbW \equiv^\textrm{sylv} UcaVbW$ where~${a < b < c}$ are letters while~$U,V,W$ are words~on~$\N$.

The objective of the present work is to discuss similar Hopf algebra structures on congruence classes of all lattice quotients of the weak order on~$\fS_n$.
Several examples of relevant combinatorial structures arise from lattice quotients of the weak order.
The fundamental example is the Tamari lattice introduced by D.~Tamari in~\cite{Tamari} and largely studied since then (see the survey book~\cite{TamariFestschrift}).
It can be defined as the transitive closure of right rotations on binary trees.
It is also (isomorphic to) the quotient of the weak order on~$\fS_n$ by the above-mentioned sylvester congruence.
See \fref{fig:weakOrderTamariLattice}.
Many other relevant lattice quotients of the weak order have been studied, see in particular~\cite{Reading-CambrianLattices, ChatelPilaud, PilaudPons-permutrees, LawReading, Giraudo, Law, Pilaud-brickAlgebra}.
N.~Reading provided in~\cite{Reading-arcDiagrams} a powerful combinatorial description of the lattice congruences of the weak order and of their congruence classes in terms of collections of certain arcs and noncrossing arc diagrams.

\begin{figure}[t]
	\capstart
	\centerline{$\vcenter{\includegraphics[scale=.8]{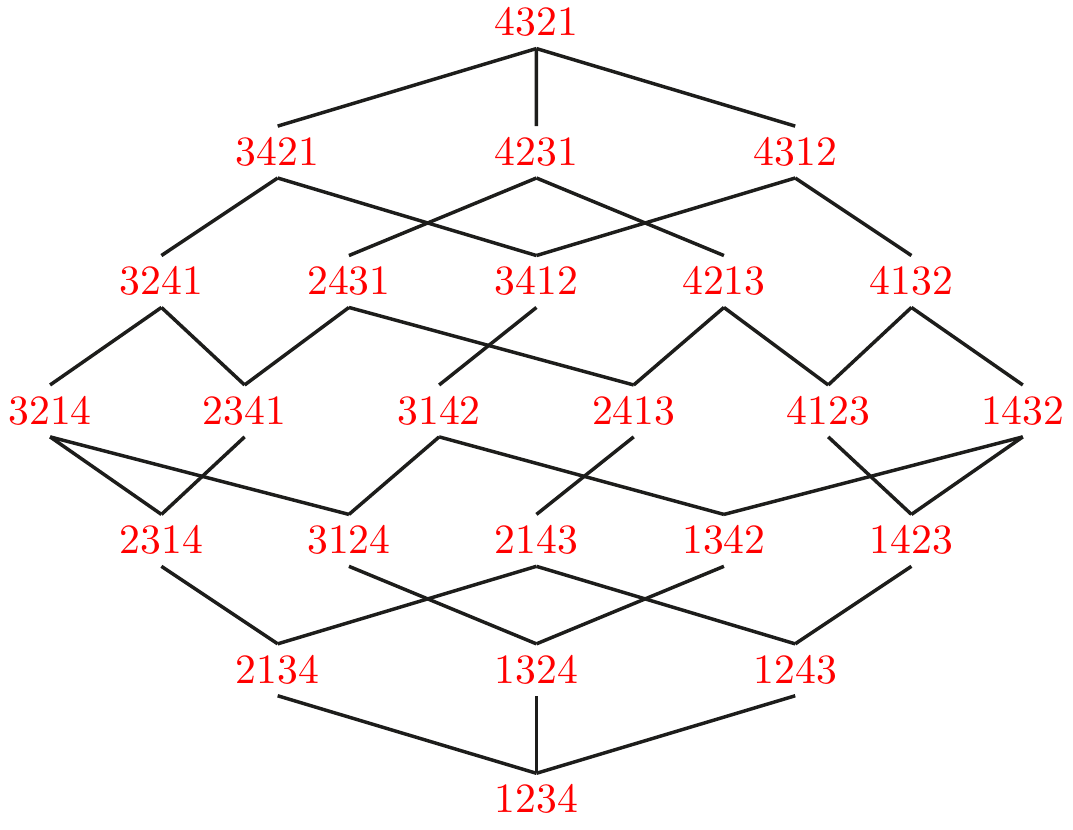}} \hspace{-7cm} \vcenter{\includegraphics[scale=.62]{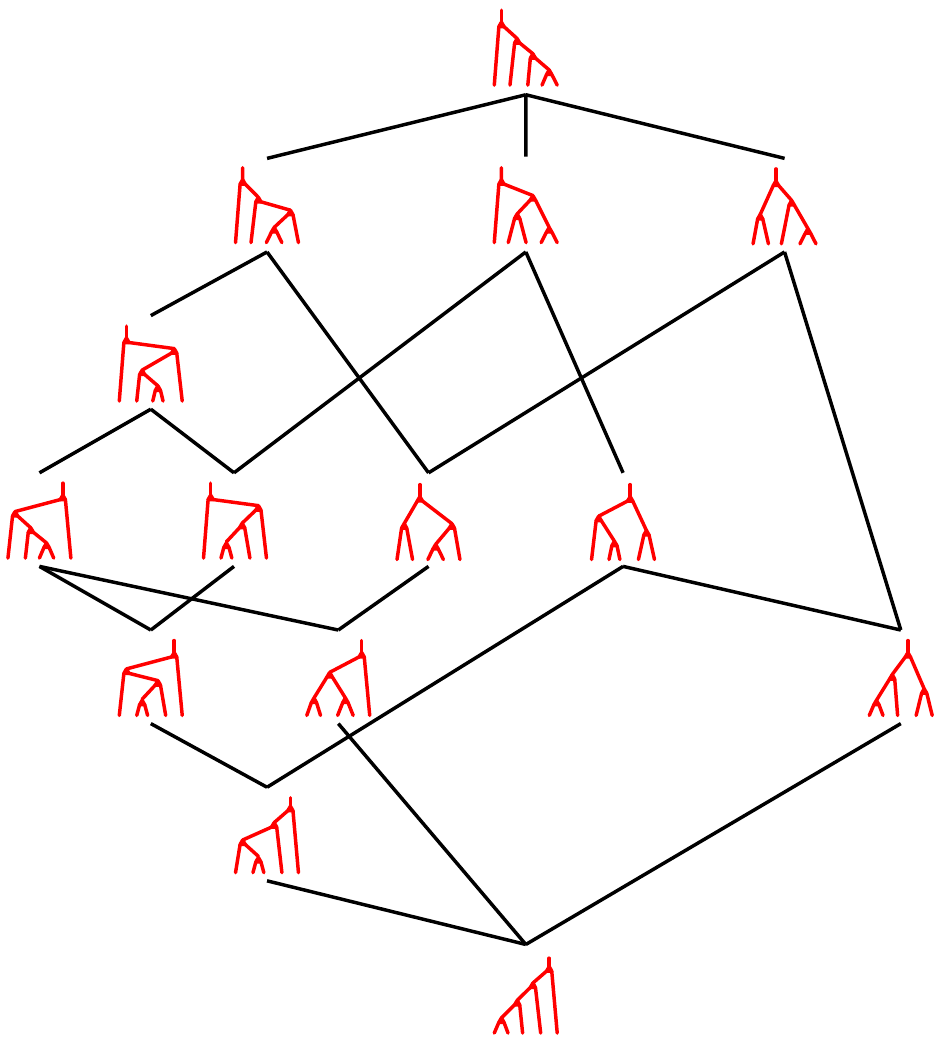}}$}
	\caption{The weak order (left) and the Tamari lattice (right).}
	\label{fig:weakOrderTamariLattice}
\end{figure}

The search for Hopf algebra structures on congruence classes of lattice quotients of the weak order was pioneered by N.~Reading.
In~\cite{Reading-HopfAlgebras}, he studied Hopf subalgebras of~$\MR$ generated by sums of permutations over the classes of a fixed lattice congruence~$\equiv_n$ on each~$\fS_n$ for~${n \ge 0}$.
He called translational (resp.~insertional) certain families~$(\equiv_n)_{n \in \N}$ of congruences that yield a subalegbra (resp.~subcoalgebra).
This approach produces relevant Hopf algebras indexed by interesting combinatorial objects such as permutations~\cite{MalvenutoReutenauer}, binary trees~\cite{LodayRonco}, diagonal rectangulations~\cite{LawReading} (or equivalently twin binary trees~\cite{Giraudo}), sashes~\cite{Law}, certain pipe dreams called twists~\cite{Pilaud-brickAlgebra}, etc.
However, the conditions on these families of congruences are rather constrained.

A more recent approach, initiated by G.~Chatel and V.~Pilaud for the Cambrian algebra~\cite{ChatelPilaud} and extended by V.~Pilaud and V.~Pons for the permutree algebra~\cite{PilaudPons-permutrees}, consists of constructing subalgebras of decorated versions of the algebra~$\MR$.
For example, \cite{ChatelPilaud} considers simultaneously all Cambrian congruences defined in~\cite{Reading-CambrianLattices}.
These congruences are given by certain rewriting rules~\cite{Reading-CambrianLattices} generalizing the sylvester congruence, and their classes are given by linear extensions of certain Cambrian trees~\cite{LangePilaud, ChatelPilaud} generalizing binary trees.
Since these congruences depend on a sequence of signs, the Cambrian algebra of~\cite{ChatelPilaud} is constructed as a subalgebra of a Hopf algebra on signed permutations generalizing~$\MR$ and studied by J.-C.~Novelli and J.-Y.~Thibon in~\cite{NovelliThibon-coloredHopfAlgebras}.
The same idea was used in~\cite{PilaudPons-permutrees} to construct an algebra on permutrees.
Note that the Hopf algebra of~\cite{NovelliThibon-coloredHopfAlgebras} on signed permutations is also useful for type~$B$ generaizations, see for instance~\cite{FoissyFromentin, JosuatNovelliThibon}.

In this paper, we explore this approach further to construct Hopf algebras on other families of congruences of the weak order.
Starting with a graded set of decorations~$\f{X}$ endowed with an operation of concatenation and an operation of selection that fulfill natural compatibility relations (see Definition~\ref{def:decorationSet}), we construct a Hopf algebra on permutations decorated with elements of~$\f{X}$.
Provided a well-chosen map~$\arcs$ from the decoration set~$\f{X}$ to the lattice congruences of the weak order (see Definition~\ref{def:conservative}), we then construct a Hopf algebra on the classes of the lattice congruences in the image of~$\arcs$.
This algebra is obtained as a Hopf subalgebra of the Hopf algebra of $\f{X}$-decorated permutations.
The choice of~$\f{X}$ and~$\arcs$ leaves quite some flexibility and allows to construct different relevant Hopf algebras on lattice congruence classes.
In this paper, we apply this general recipe in two particular settings:
\begin{enumerate}[(i)]
\item In the first setting, the image of~$\arcs$ is a family of lattice congruences of the weak order that simultaneously generalize the permutree congruences of~\cite{PilaudPons-permutrees} and the twist congruences of~\cite{Pilaud-brickAlgebra}. The resulting Hopf algebra contains (as Hopf subalgebras) those of~\cite{PilaudPons-permutrees, Pilaud-brickAlgebra}. \\[-.2cm]
\item In the second setting, the map~$\arcs$ is surjective so that we obtain a Hopf algebra involving the classes of all lattice congruences of the weak order. It contains (as Hopf subalgebras) the algebras of~\cite{MalvenutoReutenauer, LodayRonco, ChatelPilaud, PilaudPons-permutrees} but not those of~\cite{LawReading, Giraudo, Pilaud-brickAlgebra}.
\end{enumerate}
Let us mention that, although our method covers simultaneously all lattice congruences and allows quite some flexibility on the resulting Hopf algebra structure, it is restricted to lattice congruences of the weak order. Many monoid congruences are not lattice congruences of the weak order, but give rise to relevant subalgebras of~$\MR$. For instance, the plactic monoid related to the Robinson-Schensted-Knuth insertion gives rise to the Hopf algebra of S.~Poirier and C.~Reutenauer on Young tableaux~\cite{PoirierReutenauer}.

The paper is organized as follows.
In Section~\ref{sec:latticeCongruences}, we first recall N.~Reading's combinatorial model for lattice quotients of the weak order on~$\fS_n$ in terms of arc diagrams, and provide a combinatorial description of the surjection map from permutations to noncrossing arc diagrams of any fixed lattice congruence of~$\fS_n$.
In Section~\ref{sec:generalRecipe}, we present our general recipe to construct Hopf algebras on permutations and arc diagrams decorated with a given decoration set.
Finally, Section~\ref{sec:examples} is devoted to some relevant applications of this general recipe.

%%%%%%%%%%%%%%%%%%%%%%%%%%%%%%%%%%%%%%

\section{Lattice congruences of the weak order and arc diagrams}
\label{sec:latticeCongruences}

We first review a powerful combinatorial interpretation of the lattice quotients of the weak order on permutations in terms of arc diagrams.
All results presented in this section are either borrowed or directly follow from N.~Reading's work on noncrossing arc diagrams~\cite{Reading-arcDiagrams}.

%%%%%%%

\subsection{Canonical representations of permutations and noncrossing arc diagrams}
\label{subsec:canonicalRepresentations}

Consider a finite lattice~$(L,\le,\meet,\join)$.
A \defn{join representation} of~$x \in L$ is a subset~$J \subseteq L$ such that~${x = \bigJoin J}$.
Such a representation is \defn{irredundant} if~$x \ne \bigJoin J'$ for a strict subset~$J' \subsetneq J$.
The irredundant join representations of an element~$x \in L$ are ordered by containement of the lower ideals of their elements, \ie~$J \le J'$ if and only if for any~$y \in J$ there exists~$y' \in J'$ such that~$y \le y'$ in~$L$.
When this order has a minimal element, it is called the \defn{canonical join representation} of~$x$.
All elements of the canonical join representation~$x = \bigJoin J$ are then \defn{join-irreducible}, \ie cover a single element.
A lattice is \defn{join-semidistributive} when every element has a canonical join representation.
Equivalently~\cite[Thm.~2.24]{FreeseNation}, ${x \join z = y \join z \Longrightarrow x \join z = (x \meet y) \join z}$ for any~${x, y, z \in L}$.
\defn{Canonical meet representations}, \defn{meet-irreducible elements} and \defn{meet-semidistributive lattices} are defined dually.
A lattice is \defn{semi-distributive} if it is both join- and meet-semidistributive.

Let~$[n] \eqdef \{1, \dots, n\}$ and let~$[a,b] \eqdef \{a, \dots, b\}$ and~${{]a,b[} \eqdef \{a+1, \dots, b-1\}}$ for~$a < b$.
Consider the set~$\fS_n$ of permutations of~$[n]$.
An \defn{inversion} of~$\sigma = \sigma_1 \dots \sigma_n \in \fS_n$ is a pair~$(\sigma_i, \sigma_j)$ such that~${i < j}$ and~${\sigma_i > \sigma_j}$.
Denote by~$\inv(\sigma)$ the \defn{inversion set} of~$\sigma$.
The \defn{weak order} on~$\fS_n$ is defined by inclusion of inversion sets, that is $\sigma \le \tau$ if and only if ${\inv(\sigma) \subseteq \inv(\tau)}$.
Its minimal (resp.~maximal) element is the permutation~$1 \dots n$ (resp.~$n \dots 1$) and its cover relations correspond to swapping two consecutive entries in a permutation.
See \fref{fig:weakOrderTamariLattice}.
The weak order on~$\fS_n$ is known to be a semidistributive lattice.
The canonical join and meet representations of a permutation~$\sigma$ were explicitly described by N.~Reading in~\cite{Reading-arcDiagrams} as follows.

A \defn{descent} (resp.~\defn{ascent}) in~${\sigma = \sigma_1 \dots \sigma_n \in \fS_n}$ is a position~$i \in [n-1]$ such that~$\sigma_i > \sigma_{i+1}$ (resp.~${\sigma_i < \sigma_{i+1}}$).
For a descent~$i$ of~$\sigma$, define~$\underline{\lambda}(\sigma, i)$ to be the permutation whose entries are given by~$1 \dots (\sigma_{i+1} - 1)$ followed by~$\set{\sigma_j}{j < i, \, \sigma_j \in {]\sigma_{i+1}, \sigma_i[}}$ in increasing order, then~$\sigma_i \sigma_{i+1}$, then~$\set{\sigma_j}{j > i+1, \, \sigma_j \in {]\sigma_{i+1}, \sigma_i[}}$ in increasing order, and finally~$(\sigma_i + 1) \dots n$.
This permutation~$\underline{\lambda}(\sigma, i)$ is join-irreducible since it has a unique descent~$\sigma_i > \sigma_{i+1}$.
We define dually a meet-irreducible permutation~$\overline{\lambda}(\sigma, i) \eqdef \omega_\circ \underline{\lambda}(\omega_\circ\sigma,i)$ for each ascent~$i$ of~$\sigma$, where~$\omega_\circ \eqdef [n, n-1, \dots, 2, 1]$ is the longest permutation of~$\fS_n$.

\begin{theorem}[{\cite[Thm.~2.4]{Reading-arcDiagrams}}]
\label{thm:joinMeetRepresentationsPermutations}
The canonical join and meet representations of a permutation~$\sigma = \sigma_1 \dots \sigma_n$ are given by~$\bigJoin \set{\underline{\lambda}(\sigma, i)}{\sigma_i > \sigma_{i+1}}$ and~$\bigMeet \set{\overline{\lambda}(\sigma, i)}{\sigma_i < \sigma_{i+1}}$.
\end{theorem}

As N.~Reading observed in~\cite{Reading-arcDiagrams}, the permutation~$\underline{\lambda}(\sigma, i)$ is uniquely determined by the values~$\sigma_i$ and~$\sigma_{i+1}$ and by the set~$\set{\sigma_j}{j < i, \, \sigma_j \in {]\sigma_{i+1}, \sigma_i[}}$.
This combinatorial data can be recorded in the following combinatorial gadgets.

An \defn{arc} is a quadruple~$(a, b, n, S)$ where~$a,b,n \in \N$ are such that~$1 \le a < b \le n$, and~${S \subseteq {]a,b[}}$.
We define~$\c{A}_n \eqdef \set{(a, b, n, S)}{1 \le a < b \le n \text{ and } S \subseteq {]a,b[}}$, and~$\c{A} \eqdef \bigsqcup_{n \in \N} \c{A}_n$.
For a permutation~${\sigma \in \fS_n}$, we denote by~$\underline{\alpha}(i, i+1, \sigma) \eqdef (\sigma_{i+1}, \sigma_i, n, \set{\sigma_j}{j < i \text{ and } \sigma_j \in {]\sigma_{i+1}, \sigma_i[} \,})$ the arc associated to a descent~$i$ of~$\sigma$ and by~$\underline{\delta}(\sigma) \eqdef \set{\underline{\alpha}(i, i+1, \sigma)}{\sigma_i > \sigma_{i+1}}$ the set of arcs corresponding to all descents of~$\sigma$. We define~$\overline{\alpha}$ and~$\overline{\delta}$ dually for ascents.

An arc~$(a, b, n, S)$ can be visually represented as an $x$-monotone continuous curve wiggling around the horizontal axis, with endpoints~$a$ and~$b$, and passing above the points of~$S$ and below the points of~${]a,b[} \ssm S$.
Using this representation, N.~Reading provided a convenient visual interpretation of~$\underline{\delta}$ and~$\overline{\delta}$.
For this, represent the permutation~$\sigma$ by its \defn{permutation table}~$(\sigma_i, i)$.
(This unusual choice of orientation is necessary to fit later with the existing constructions ~\cite{LodayRonco, HivertNovelliThibon-algebraBinarySearchTrees, ChatelPilaud, PilaudPons-permutrees}.)
Draw arcs between any two consecutive dots~$(\sigma_i, i)$ and~$(\sigma_{i+1}, i+1)$, colored green if~$\sigma_i < \sigma_{i+1}$ is an ascent and red if~$\sigma_i > \sigma_{i+1}$ is a descent.
Then move all dots down to the horizontal axis, allowing the segments to curve, but not to cross each other nor to pass through any dot.
The set of red (resp.~green) arcs is then the set~$\underline{\delta}(\sigma)$ (resp.~$\overline{\delta}(\sigma)$) corresponding to the canonical join (resp.~meet) representation of~$\sigma$.
See \fref{fig:noncrossingArcDiagrams} for illustrations of these maps.
%We denote by~$\delta(\sigma)$ the bicolored set of arcs~$(\underline{\delta}(\sigma), \overline{\delta}(\sigma))$.

\begin{figure}
	\capstart
	\centerline{\includegraphics[scale=.85]{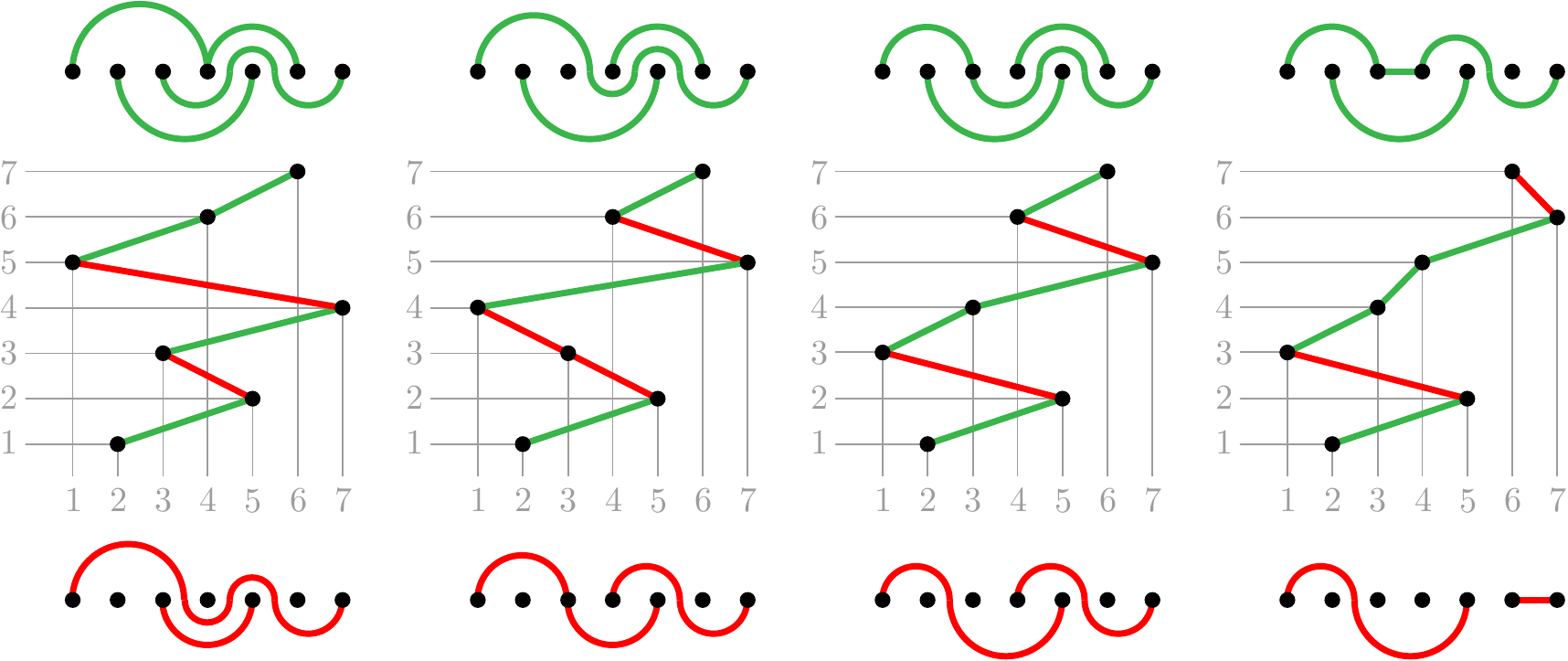}}
	\caption{The noncrossing arc diagrams~$\underline{\delta}(\sigma)$ (bottom) and~$\overline{\delta}(\sigma)$ (top) for the permutations~$\sigma = 2537146$, $2531746$, $2513746$, and $2513476$.}
	\label{fig:noncrossingArcDiagrams}
\end{figure}

This representation provides a natural characterization of the sets of join-irreducible (resp.~meet-irreducible) permutations that form canonical join (resp.~meet) representations.
We say that two arcs \defn{cross} if the interior of the two curves representing these arcs intersect.
In other words, the arcs~$(a, b, n, R)$ and~$(c, d, n, S)$ cross if there exist~$r, s \in \big( [a,b] \cap [c,d] \big) \ssm \big( \{a,b\} \cap \{c,d\} \big)$ such that~$r \in \big( R \cup \{a,b\} \big) \ssm S$ while $s \in \big( S \cap \{c,d\} \big) \ssm R$.
A \defn{noncrossing arc diagram} is a collection~$\c{D}$ of arcs of~$\c{A}_n$ such that for any two arcs~$\alpha, \beta \in \c{D}$ do not cross and have distinct left endpoints and distinct right endpoints (but the right endpoint of~$\alpha$ can be the left endpoint of~$\beta$ or~\mbox{\viceversa}).
See \fref{fig:noncrossingArcDiagrams} for examples of noncrossing arc diagrams.

\begin{theorem}[{\cite[Thm.~3.1]{Reading-arcDiagrams}}]
The maps~$\underline{\delta}$ and~$\overline{\delta}$ are bijections from permutations of~$\fS_n$ to noncrossing arc diagrams of~$\c{A}_n$.
\end{theorem}

The reverse bijections~$\underline{\delta}^{-1}$ and~$\overline{\delta}^{-1}$ are explicitly described in~\cite[Prop.~3.2]{Reading-arcDiagrams}.
Briefly speaking, consider the poset of connected components of~$\c{D}$ ordered by (the transitive closure of) the priority $X \prec Y$ if there is an arc~$\alpha = (a, b, n, S) \in \c{D}$ with~$S \cap X \ne \varnothing$ and~$a, b \in Y$ or with~$a, b \in X$ and~$({]a,b[} \ssm S) \cap Y \ne \varnothing$.
To obtain~$\underline{\delta}{}^{-1}(\c{D})$ (resp.~$\overline{\delta}{}^{-1}(\c{D})$), choose the linear extension of this priority poset where ties are resolved by choosing first the leftmost (resp.~rightmost) connected component, and order decreasingly (resp.~increasingly) the values in each connected component.
See \fref{fig:noncrossingArcDiagrams}.

%%%%%%%

\subsection{Lattice quotients of the weak order}
\label{subsec:latticeQuotients}

Consider a finite lattice~${(L,\le,\meet,\join)}$. A \defn{lattice congruence} of~$L$ is an equivalence relation on~$L$ that respects the meet and the join operations, \ie such that $x \equiv x'$ and~$y \equiv y'$ implies ${x \meet y \, \equiv \, x' \meet y'}$ and~${x \join y \, \equiv \, x' \join y'}$.
Equivalently, the equivalence classes of~$\equiv$ are intervals of~$L$, and the up and down maps~$\projUp$ and~$\projDown$, respectively sending an element of~$L$ to the top and bottom elements of its equivalence class for~$\equiv$, are order-preserving.
A lattice congruence~$\equiv$ defines a \defn{lattice quotient}~$L/{\equiv}$ on the congruence classes of~$\equiv$ where the order relation is given by~$X \le Y$ if and only if there exists~$x \in X$ and~$y \in Y$ such that~$x \le y$.
The meet~$X \meet Y$ (resp.~the join~$X \join Y$) of two congruence classes~$X$ and~$Y$ is the congruence class of~$x \meet y$ (resp.~of~$x \join y$) for arbitrary representatives~$x \in X$~and~$y \in Y$.
Intuitively, the quotient~$L/{\equiv}$ is obtained by contracting the equivalence classes of~$\equiv$ in the lattice~$L$.
More precisely, we say that an element~$x$ is \defn{contracted} by~$\equiv$ if it is not minimal in its equivalence class of~$\equiv$, \ie if~$x \ne \projDown(x)$.
As each class of~$\equiv$ is an interval of~$L$, it contains a unique uncontracted element, and the quotient~$L/{\equiv}$ is isomorphic to the subposet of~$L$ induced by its uncontracted elements.
Moreover, the canonical join representations in the quotient~$\projDown(L)$ are precisely the canonical join representations of~$L$ that do not involve any contracted join-irreducible.
This yields the following.

\begin{theorem}[{\cite[Thm.~4.1]{Reading-arcDiagrams}}]
\label{thm:arcDiagrams}
Consider a lattice congruence~$\equiv$ of the weak order on~$\fS_n$, and let~$\c{I}_\equiv$ denote the arcs corresponding to the join-irreducible permutations not contracted by~$\equiv$.
\begin{enumerate}[(i)]
\item A permutation~$\sigma$ is minimal in its $\equiv$-congruence class if and only if~${\underline{\delta}(\sigma) \subseteq \c{I}_\equiv}$.
\item Sending a $\equiv$-congruence class with minimal permutation~$\sigma$ to the arc diagram~$\underline{\delta}(\sigma)$ defines a bijection between the congruence classes of~$\equiv$ and the noncrossing arc diagrams~in~$\c{I}_\equiv$.
\item The congruence~$\equiv$ is the transitive closure of the rewriting rule~$\sigma \to \sigma \cdot (i \; i+1)$ where~$i$ is a descent of~$\sigma$ such that~${\underline{\alpha}(i, i+1, \sigma)} \notin \c{I}_\equiv$.
\end{enumerate}
\end{theorem}

A dual statement holds replacing minimal by maximal and~$\underline{\delta}(\sigma)$ by~$\overline{\delta}(\sigma)$.
Moreover, the sets of arcs~$\bigset{\underline{\delta}(\sigma)}{\sigma \text{ join-irreducible}, \; \projDown(\sigma) = \sigma}$ and~$\bigset{\overline{\delta}(\sigma)}{\sigma \text{ meet-irreducible}, \; \projUp(\sigma) = \sigma}$ coincide.
We denote still this set~$\c{I}_\equiv$.

%%%%%%%

\subsection{Arc ideals}
\label{subsec:arcIdeals}

It remains to characterize the sets of arcs~$\c{I}_\equiv$ corresponding to the uncontracted join-irreducibles of a lattice congruence of the weak order.
This is again transparent on the arc representation of join-irreducible permutations.
An arc~$(a, d, n, S)$ is \defn{forced} by an arc~$(b, c, n, T)$ if~${a \le b < c \le d}$ and~${T = S \cap {]b,c[}}$. 
Graphically, it means that~$(b, c, n, T)$ is obtained by restricting the arc~$(a, d, n, S)$ to the interval~$[b,c]$.
We denote by~$(a, d, n, S) \prec (b, c, n, T)$ the forcing order.
\fref{fig:forcingOrderArcIdeals}\,(left) shows this order on~$\c{A}_4$.

%\begin{figure}
%	\capstart
%	\centerline{\includegraphics[scale=.8]{forcingOrder}}
%	\caption{The forcing order on arcs of~$\c{A}_4$.}
%	\label{fig:forcingOrder}
%\end{figure}

\begin{figure}
	\capstart
	\centerline{\includegraphics[scale=.55]{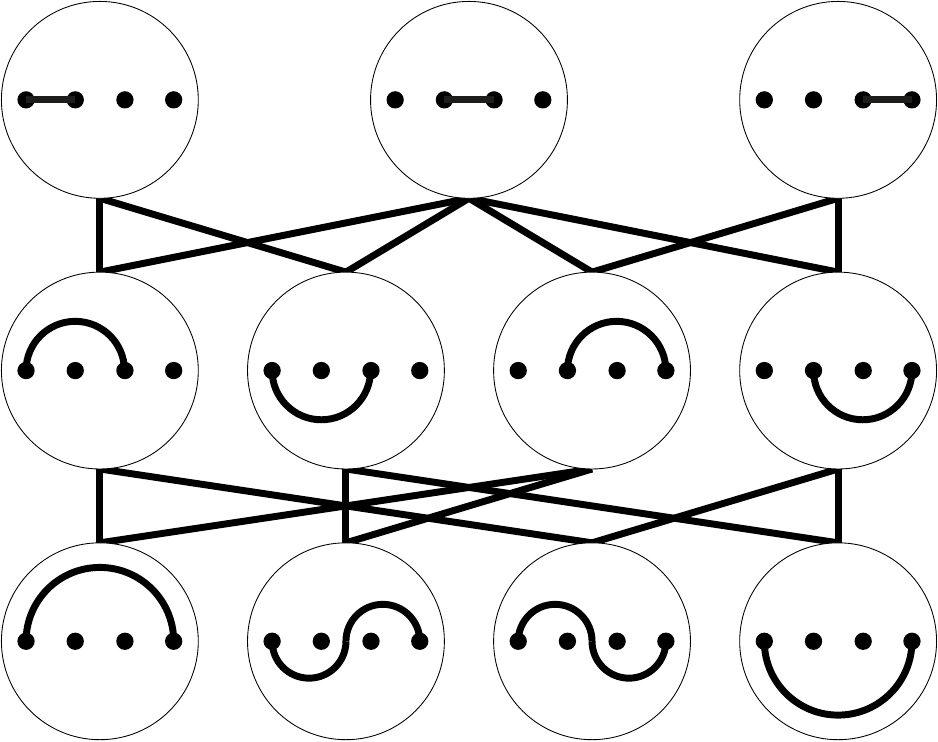} \hspace{1.5cm} \includegraphics[scale=.9]{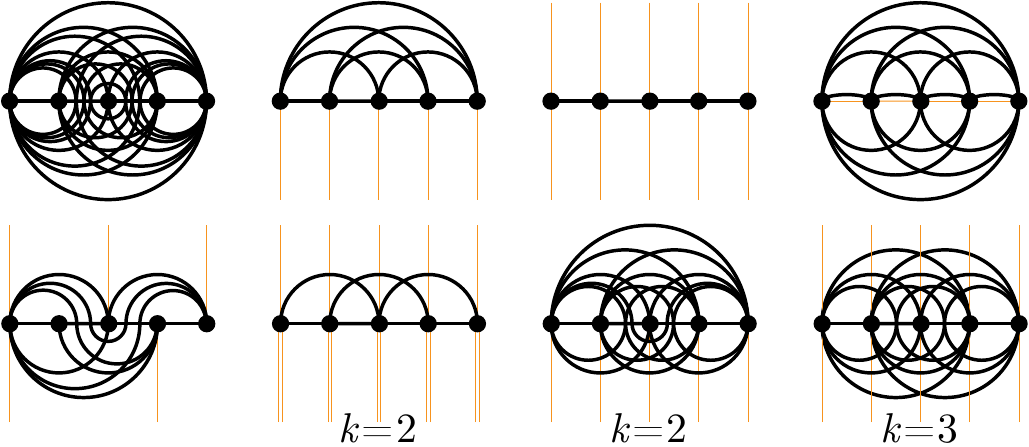}}
	\caption{The forcing order on arcs of~$\c{A}_4$ (left) and some examples of arc ideals~$\c{A}_{\north, \south, \east, \west}^{< k}$ (right) whose associated lattice congruence classes correspond to permutations, binary trees, binary sequences, diagonal rectangulations, permutrees, sashes, acyclic $2$-twists, and $3$-descent schemes. Walls are in red and $k = 1$ if not stated otherwise.}
	\label{fig:forcingOrderArcIdeals}
\end{figure}

\begin{theorem}[{\cite[Coro.~4.5]{Reading-arcDiagrams}}]
\label{thm:arcIdeals}
A set of arcs~$\c{I} \subseteq \c{A}_n$ corresponds to the set of uncontracted join-irreducible permutations of some lattice congruence~$\equiv$ of the weak order on~$\fS_n$ if and only if it is an upper ideal of the forcing order~$\prec$.
\end{theorem}

Call \defn{arc ideal} any upper ideal~$\c{I}$ of the forcing order: $(a, d, n, S) \in \c{I}$ implies~${(b, c, n, S \cap {]b,c[}) \in \c{I}}$ for all~$a \le b < c \le d$ and~$S \subseteq {]a,d[}$.
We denote by~$\f{I}_n$ the set of arc ideals of~$\c{A}_n$.

\begin{example}
\label{exm:boundedCrossingCongruences}
The sets of all arcs~$\c{A}_n$, the set of upper arcs~$\c{A}_n^+ \eqdef \set{(a, b, n, {]a, b[})}{1 \le a < b \le n}$, the set of lower arcs~$\c{A}_n^- \eqdef \set{(a, b, n, \varnothing)}{1 \le a < b \le n}$, or the union~$\c{A}_n^+ \cup \c{A}_n^-$ are all arc ideals.
More generally, fix four functions~$\north, \south, \east, \west : [n] \to \N$ and choose~$k \in \N$.
For each~$a \in [n]$, draw $\north(a)$ upper vertical walls above~$a$, $\south(a)$ lower vertical walls below~$a$, and~$\min(\east(a), \west(a+1))$ horizontal walls from~$a$ and~$a+1$.
Then the set~$\c{A}_{\north, \south, \east, \west}^{< k}$ of arcs that cross at most~$k-1$ of all these walls is an arc ideal.
For certain choices of~$\north, \south, \east, \west$ and~$k$, the resulting arc ideals can correspond to:
\begin{itemize}
\item the weak order~(${\north = \south = \east = \west = 0}$ and~$k = 1$), 
\item the Tamari lattice~\cite{Tamari} ($\north = \east = \west = 0$ and ${\south = k = 1}$), 
\item the boolean lattice ($\north = \south = k = 1$ and $\east = \west = 0$), 
\item the lattice of diagonal rectangulations~\cite{LawReading}~($\north = \south = 0$ and $\east = \west = k = 1$), 
\item the permutree lattices~\cite{PilaudPons-permutrees}~($\north \le 1$, $\south \le 1$, $\east = \west = 0$ and ${k = 1}$), 
\item the lattice of sashes~\cite{Law} ($\north = 1$, $\south = 2$, $\east = \west = 0$ and~$k = 2$), 
\item the lattice of acyclic $k$-twists~\cite{Pilaud-brickAlgebra} ($\north = \east = \west = 0$, ${\south = 1}$ and~$k \ge 1$), 
\item the lattice of $k$-descent schemes~\cite{NovelliReutenauerThibon, Pilaud-brickAlgebra} ($\north = \south = 1$, $\east = \west = 0$ and~$k \ge 1$).
\end{itemize}
See \fref{fig:forcingOrderArcIdeals}\,(right).
%
%\begin{figure}
%	\capstart
%	\centerline{\includegraphics[scale=1.1]{arcIdeals}}
%	\caption{Examples of arc ideals~$\c{A}_{\north, \south, \east, \west}^{< k}$. The associated lattice congruences correspond from left to right to permutations, binary trees, diagonal rectangulations, permutrees, sashes, and acyclic $k$-twists.}
%	\label{fig:arcIdeals}
%\end{figure}
\end{example}

%%%%%%%

\subsection{Explicit surjection}
\label{subsec:surjection}

Consider a lattice congruence~$\equiv$ of the weak order and let~$\c{I} = \c{I}_\equiv$.
According to Theorem~\ref{thm:arcDiagrams}, the congruence classes of~$\equiv$ are in bijection with noncrossing arc diagrams of~$\c{I}$.
Moreover, the map~$\underline{\eta}{}_\c{I}$ defined by~$\underline{\eta}{}_\c{I}(\sigma) \eqdef \underline{\delta}\big( \projDown(\sigma) \big)$ sends a permutation~$\sigma \in \fS_n$ to the noncrossing arc diagram of~$\c{I}$ corresponding to the congruence class of~$\sigma$.
For completeness, we now provide a direct explicit description of this surjection~$\underline{\eta}{}_\c{I}$.

\begin{figure}[b]
	\capstart
	\centerline{\includegraphics[scale=.85]{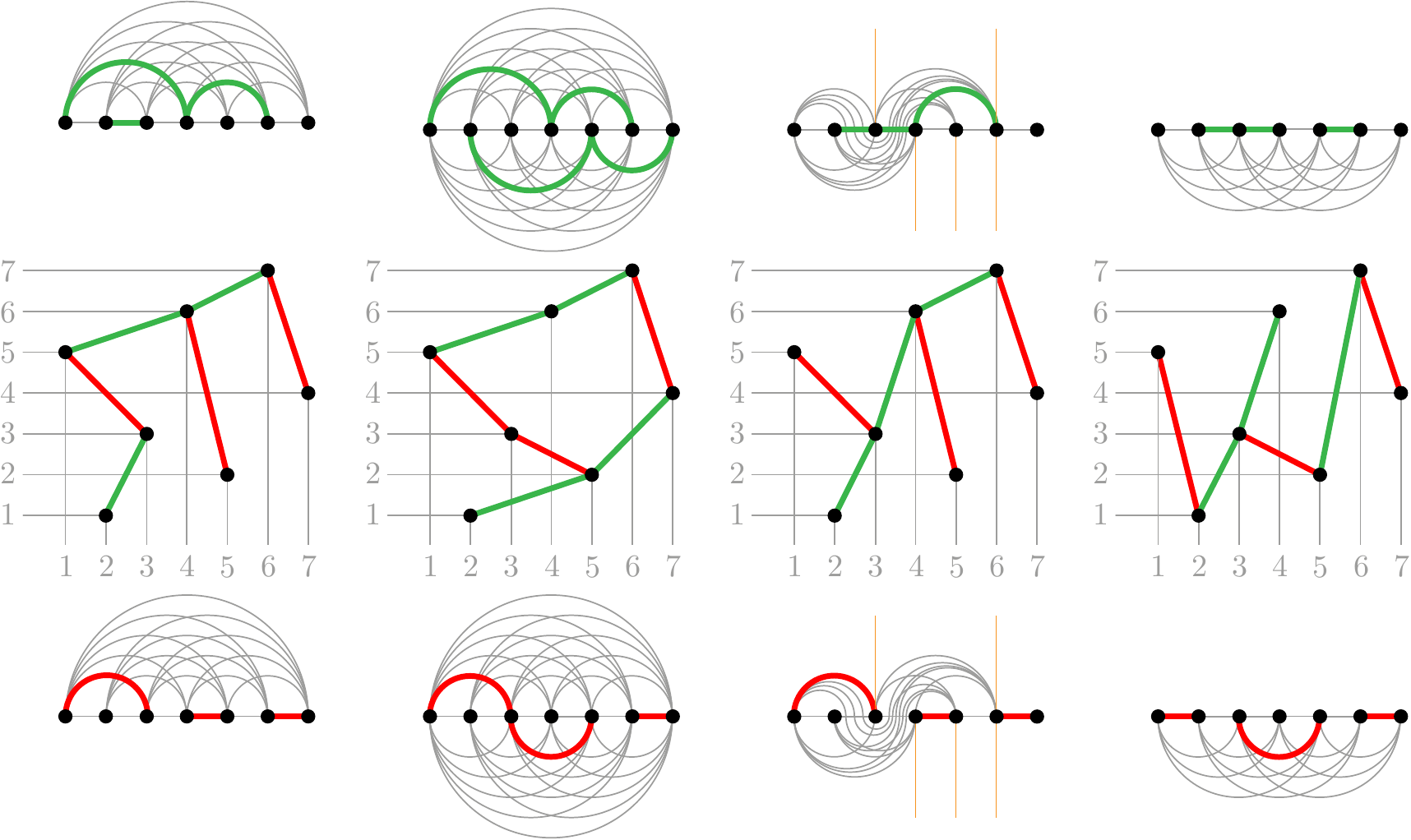}}
	\caption{The noncrossing arc diagrams~$\underline{\eta}_\c{I}(\sigma)$ (bottom) and~$\overline{\eta}_\c{I}(\sigma)$ (top) for the permutation~$\sigma = 2537146$ and different arc ideals (represented in light gray).}
	\label{fig:arcDiagramsQuotients}
\end{figure}

We still represent the permutation~$\sigma$ by its permutation table~$\set{(\sigma_i, i)}{i \in [n]}$.
We define
\[
\boxbslash(\sigma) \eqdef \set{(i, j)}{1 \le i < j \le n, \, \sigma_i > \sigma_j \text{ and } \sigma({]i,j[}) \cap {]\sigma_j, \sigma_i[} = \varnothing}.
\]
Intuitively, $\boxbslash(\sigma)$ are the pairs of positions such that the rectangle with bottom right corner~$(\sigma_i, i)$ and top left corner~$(\sigma_j, j)$ contains no other point~$(\sigma_k, k)$ of the permutation table of~$\sigma$.
Order~$\boxbslash(\sigma)$ by~$(i, j) \prec (k,\ell)$ if~$i \le k < \ell \le j$ and~$\sigma_k \ge \sigma_i > \sigma_j \ge \sigma_\ell$.
For~$(i, j) \in \boxbslash(\sigma)$, we define an arc~$\alpha(i, j, \sigma)$ by
\[
\underline{\alpha}(i, j, \sigma) \eqdef (\sigma_j, \sigma_i, n, \set{\sigma_k}{j < k \text{ and } \sigma_k \in {]\sigma_j, \sigma_i[} \,}).
\]
Note that it extends our previous definition of~$\underline{\alpha}(i, i+1, \sigma)$ in Section~\ref{subsec:canonicalRepresentations}.
Finally, define~$\boxbslash_\c{I}(\sigma)$ to be the subset of $\prec$-maximal elements in~$\set{(i, j) \in \boxbslash(\sigma)}{\underline{\alpha}(i, j, \sigma) \in \c{I}}$.

\begin{proposition}
For any~$\sigma \in \fS_n$, the set~$\underline{\eta}{}_\c{I}(\sigma) = \underline{\delta}\big( \projDown(\sigma) \big) = \bigset{\underline{\alpha}(i, j, \sigma)}{(i, j) \in \boxbslash_\c{I}(\sigma)}$ is the noncrossing arc diagram of~$\c{I}$ corresponding to the $\equiv$-congruence class of~$\sigma$.
\end{proposition}

\begin{proof}
We first observe that if~$\sigma \in \fS_n$ is minimal in its congruence class, then~$\underline{\eta}{}_\c{I}(\sigma) = \underline{\delta}(\sigma)$ is the arc diagram of the congruence class of~$\sigma$.
Indeed, observe first that the characterization of Theorem~\ref{thm:arcDiagrams}\,(i) ensures that~$\boxbslash_\c{I}(\sigma)$ contains~$(i,i+1)$ for each descent~$i$ of~$\sigma$.
Conversely, consider~$(i, j) \in \boxbslash_\c{I}(\sigma)$.
Since~$(i, j) \in \boxbslash(\sigma)$, there exists~$i \le k < j$ such that~$\sigma_{k+1} \le \sigma_j < \sigma_i \le \sigma_k$.
Therefore, $(k, k+1) \in \boxbslash_\c{I}(\sigma)$ so that~$(i, j) = (k, k+1)$ by $\prec$-maximality in the definition of~$\boxbslash_\c{I}(\sigma)$.

Assume now that~$\sigma \in \fS_n$ is not minimal in its congruence class.
By~Theorem~\ref{thm:arcDiagrams}\,(i), $\sigma$ has a descent~$i$ such that~$\underline{\alpha}(i, i+1, \sigma) \notin \c{I}$.
Let~$\sigma' \eqdef \sigma \cdot (i \; i+1)$.
Note that~${\boxbslash(\sigma) \ssm \{(i,i+1)\} \subseteq \boxbslash(\sigma')}$.
Conversely, consider~$(k,\ell) \in \boxbslash(\sigma') \ssm \boxbslash(\sigma)$.
Then either~$k = i$ and~$\ell > j$, or~$k < i$ and~$\ell = j$.
In both cases, $\underline{\alpha}(i, j,\sigma')$ cannot belong to~$\c{I}$ since it forces~$\underline{\alpha}(i, i+1, \sigma)$ which does not belong to~$\c{I}$.
We conclude that~$\boxbslash_\c{I}(\sigma') = \boxbslash_\c{I}(\sigma)$ so that~$\underline{\eta}{}_\c{I}(\sigma') = \underline{\eta}{}_\c{I}(\sigma)$.
In other words, $\underline{\eta}{}_\c{I}(\sigma)$ is preserved by the rewriting rule of Theorem~\ref{thm:arcDiagrams}\,(iii).
This concludes the proof since this rewriting rule terminates on a permutation which is minimal in its congruence class.
\end{proof}

As in Section~\ref{subsec:canonicalRepresentations}, $\underline{\eta}{}_\c{I}(\sigma)$ is obtained by connecting the points~$(\sigma_i, i)$ and~$(\sigma_j, j)$ in the table of the permutation~$\sigma$ for all~$(i, j) \in \boxbslash_\c{I}(\sigma)$, and by moving all numbers of this table to the horizontal axis, allowing the segment connecting~$i$ and~$j$ to curve but not to pass through any number.
Note that to obtain all pairs~$(i, j) \in \boxbslash_\c{I}(\sigma)$, one can either draw all pairs~$(i, j) \in \boxbslash(\sigma)$ for which~$\underline{\alpha}(i, j,\sigma) \in \c{I}$ and conserve the $\prec$-maximal ones, or one can perform a direct insertion algorithm similar to that of~\cite{ChatelPilaud, PilaudPons-permutrees}.
Details are left to the reader.
We define dually the sets~$\boxslash(\sigma)$ and~$\boxslash_\c{I}(\sigma)$ and the maps~$\overline{\alpha}$ and~$\overline{\eta}_\c{I}$.
See \fref{fig:arcDiagramsQuotients} for an illustration of the maps~$\underline{\eta}{}_\c{I}$ and~$\overline{\eta}_\c{I}$.

\section{Hopf algebra structures on noncrossing arc diagrams}
\label{sec:generalRecipe}

In this section, we present general methods to construct Hopf algebra structures on decorated permutations and decorated arc diagrams.
Recall that a combinatorial Hopf algebra is a combinatorial vector space~$\c{A}$ endowed with an associative product~$\product : \c{A} \otimes \c{A} \to \c{A}$ and a coassociative coproduct~$\coproduct : \c{A} \to \c{A} \otimes \c{A}$ so that the diagram

\centerline{
\begin{tikzpicture}
  \matrix (m) [matrix of math nodes, row sep=2.5em, column sep=5em, minimum width=2em, ampersand replacement=\&]
  {
	\c{A} \otimes \c{A}  									\& \c{A}	\& \c{A} \otimes \c{A}									\\
	\c{A} \otimes \c{A} \otimes \c{A} \otimes \c{A}	\& 			\& \c{A} \otimes \c{A} \otimes \c{A} \otimes \c{A}	\\
  };
  \path[->]
    (m-1-1) edge node [above] {$\product$} (m-1-2)
            edge node [left] {$\coproduct \otimes \coproduct$} (m-2-1)
    (m-1-2) edge node [above] {$\coproduct$} (m-1-3)
    (m-2-1) edge node [below] {$I \otimes \swap \otimes I$} (m-2-3)
    (m-2-3) edge node [right] {$\product \otimes \product$} (m-1-3);
\end{tikzpicture}
}

\noindent
commutes, where~$\swap : \c{A} \otimes \c{A} \to \c{A} \otimes \c{A}$ is defined by~$\swap(x \otimes y) = y \otimes x$ and~$I$ is the identity.

%%%%%%%

\subsection{Hopf algebra on permutations}
\label{subsec:MalvenutoReutenauer}

Before constructing our decorated versions, we briefly recall C.~Malvenuto and C.~Reutenauer's Hopf algebra on permutations~\cite{MalvenutoReutenauer}.
We denote by~$\fS \eqdef \bigsqcup_{n \in \N} \fS_n$ the set of all permutations, of arbitrary size.

The \defn{standardization} of a word~$w \in \N^q$ with distinct entries is the permutation~$\stdf(w)$ of~$[q]$ whose entries are in the same relative order as the entries of~$w$.
For a permutation~$\rho \in \fS_p$ and a subset~$R = \{r_1 < \dots < r_q\} \subseteq [p]$, we define $\stdpos{\rho}{R}$ (resp.~$\stdval{\rho}{R}$) as the standardization of the word obtained by deleting form~$\rho$ the entries whose positions (resp.~values) are not in~$R$.
For two permutations~$\sigma \in \fS_m$ and~$\tau \in \fS_n$, define the \defn{shifted shuffle}~$\sigma \shiftedShuffle \tau$ and the \defn{convolution}~$\sigma \convolution \tau$~by
\begin{align*}
\sigma \shiftedShuffle \tau & \eqdef \set{\rho \in \fS_{m+n}}{\stdval{\rho}{[m]} = \sigma \text{ and } \stdval{\rho}{[m+n] \ssm [m]} = \tau} \\ 
\text{and}\qquad
\sigma \convolution \tau & \eqdef \set{\rho \in \fS_{m+n}}{\stdpos{\rho}{[m]} = \sigma \text{ and } \stdpos{\rho}{[m+n] \ssm [m]} = \tau}.
\end{align*}

For example,
\begin{align*}
{\red 12} \shiftedShuffle {\blue 231} & = \{ {\red 12}{\blue 453}, {\red 1}{\blue 4}{\red 2}{\blue 53}, {\red 1}{\blue 45}{\red 2}{\blue 3}, {\red 1}{\blue 453}{\red 2}, {\blue 4}{\red 12}{\blue 53}, {\blue 4}{\red 1}{\blue 5}{\red 2}{\blue 3}, {\blue 4}{\red 1}{\blue 53}{\red 2}, {\blue 45}{\red 12}{\blue 3}, {\blue 45}{\red 1}{\blue 3}{\red 2}, {\blue 453}{\red 12} \}, \\
\text{and}\qquad
{\red 12} \convolution {\blue 231} & = \{ {\red 12}{\blue 453}, {\red 13}{\blue 452}, {\red 14}{\blue 352}, {\red 15}{\blue 342}, {\red 23}{\blue 451}, {\red 24}{\blue 351}, {\red 25}{\blue 341}, {\red 34}{\blue 251}, {\red 35}{\blue 241}, {\red 45}{\blue 231} \}.
\end{align*}

\begin{theorem}[\cite{MalvenutoReutenauer}]
\label{thm:MalvenutoReutenauer}
The vector space~$\b{k}\fS$ with basis~$(\F_\sigma)_{\sigma \in \fS}$ endowed with the product and coproduct defined by
\[
\F_\sigma \product \F_\tau = \sum_{\rho \in \sigma \shiftedShuffle \tau} \F_\rho
\qquad\text{and}\qquad
\coproduct \F_\rho = \sum_{\rho \in \sigma \convolution \tau} \F_\sigma \otimes \F_\tau
\]
is a graded Hopf algebra.
\end{theorem}

Recall that the product in~$\b{k}\fS$ behaves nicely with the weak order on~$\fS_n$.
For two permutations~$\sigma \in \fS_m$ and~$\tau \in \fS_n$, consider the permutations~$\sigma\backslash\tau$ and~$\tau\slash\sigma$ of~$\fS_{m+n}$ defined by
\[
\sigma\backslash\tau(i) = \begin{cases} \mu(i) & \text{ if } i \in [m] \\ m + \tau(i-m) & \text{ otherwise} \end{cases}
\qquad\text{and}\qquad
\tau\slash\sigma(i) = \begin{cases} m + \tau(i) & \text{ if } i \in [n] \\ \sigma(i-n) & \text{ otherwise.} \end{cases} 
\]
The shifted shuffle~$\sigma \shiftedShuffle \tau$ is then precisely given the weak order interval between~$\sigma\backslash\tau$ and~$\sigma\slash\tau$ in the weak order on~$\fS_{m+n}$.
This extends to product of weak order intervals as follows.

\begin{proposition}
\label{prop:productIntervalWeakOrder}
A product of weak order intervals in~$\b{k}\fS$ is a weak order interval: for any two weak order intervals~$[\mu, \nu] \subseteq \fS_m$ and~$[\lambda, \omega] \subseteq \fS_n$, we have
\[
\bigg( \sum_{\mu \le \sigma \le \nu} \F_\sigma \bigg) \product \bigg( \sum_{\lambda \le \tau \le \omega} \F_\tau \bigg) = \sum_{\mu\backslash\lambda \le \rho \le \omega\slash\nu} \F_\rho.
\]
\end{proposition}

%\begin{corollary}
%For~$\tau \in \fS_n$, define
%\[
%\EFQSym^\tau = \sum_{\tau \le \tau'} \F_{\tau'}
%\qquad\text{and}\qquad
%\HFQSym^\tau = \sum_{\tau' \le \tau} \F_{\tau'}
%\]
%where~$\le$ is the weak order on~$\fS_n$. Then~$(\EFQSym_\tau)_{\tau \in \fS}$ and~$(\HFQSym_\tau)_{\tau \in \fS}$ are multiplicative bases of~$\b{k}\fS$:
%\[
%\EFQSym^\tau \product \EFQSym^{\tau'} = \EFQSym^{\underprod{\tau}{\tau'}}
%\qquad\text{and}\qquad
%\HFQSym^\tau \product \HFQSym^{\tau'} = \HFQSym^{\overprod{\tau}{\tau'}},
%\]
%where~$\underprod{\tau}{\tau'} = \tau\bar\tau'$ and~$\overprod{\tau}{\tau'} = \bar\tau'\tau$.
%A permutation~$\tau \in \fS_n$ is $\EFQSym$-decomposable (resp.~$\HFQSym$-decomposable) if and only if there exists~${k \in [n-1]}$ such that~$\tau([k]) = [k]$ (resp.~such that~$\tau([k]) = [n] \ssm [k]$). Moreover, $\b{k}\fS$ is freely generated by the elements~$\EFQSym^\tau$ (resp.~$\HFQSym^\tau$) for the $\EFQSym$-indecomposable (resp.~$\HFQSym$-indecomposable) permutations.
%\end{corollary}

%%%%%%%

\subsection{Decorated permutations}

For our purposes, we need extensions of C.~Malvenuto and C.~Reutenauer's Hopf algebra on permutations.
For example, we needed the signed or decorated permutations of~\cite{NovelliThibon-coloredHopfAlgebras} to construct the Cambrian and permutree Hopf algebras~\cite{ChatelPilaud, PilaudPons-permutrees}.

We now define Hopf algebras on permutations decorated with potentially more complicated structures.

\begin{definition}
\label{def:decorationSet}
A \defn{decoration set} is a graded set~$\f{X} \eqdef \bigsqcup_{n \ge 0} \f{X}_n$ endowed with
\begin{itemize}
\item a \defn{concatenation}~$\concatf : \f{X}_m \times \f{X}_n \longrightarrow \f{X}_{m+n}$ for all~$m, n \in \N$,
\item a \defn{selection}~$\selectf : \f{X}_m \times \binom{[m]}{k} \longrightarrow \f{X}_k$ for all~$m, k \in \N$,
\end{itemize}
such that
\begin{enumerate}[(i)]
\item $\concat{\c{X}}{\concat{\c{Y}}{\c{Z}}} = \concat{\concat{\c{X}}{\c{Y}}}{\c{Z}}$ for any elements~$\c{X}, \c{Y}, \c{Z} \in \f{X}$,
\item $\select{\select{\c{X}}{R}}{S} = \select{\c{X}}{\set{r_s}{s \in S}}$ for any element~$\c{X} \in \f{X}_p$, and any subsets $R = \{r_1, \dots, r_q\} \subseteq [p]$ and~${S \subseteq [q]}$,
\item ${\concat{\select{\c{X}}{R}}{\select{\c{Y}}{S}} = \select{\concat{\c{X}}{\c{Y}}}{R \cup S^{\rightarrow m}}}$ for any elements~${\c{X} \in \f{X}_m}$ and~${\c{Y} \in \f{X}_n}$, and any subsets~$R \subseteq [m]$ and~$S \subseteq [n]$, where~$S^{\rightarrow m} \eqdef \set{s+m}{s \in S}$.
\end{enumerate}
\end{definition}

\begin{example}
\label{exm:decorationSet}
A typical decoration set is the set of words~$\c{A}^*$ on a finite alphabet~$\c{A}$, graded by their length, with the classical concatenation of words, and the selection defined by subwords: $\concat{u_1 \dots u_m}{v_1 \dots v_n} = u_1 \dots u_mv_1 \dots v_n$ and $\select{w_1 \dots w_p}{\{r_1, \dots, r_q\}} = w_{r_1} \dots w_{r_q}$.
Among many other examples, let us also mention the set of labeled graphs, graded by their number of vertices, with the concatenation defined as the shifted union, and the selection defined by standardized induced subgraphs.
Further examples will appear in Section~\ref{sec:examples}.
\end{example}

For~$n \ge 0$, we denote by~$\eP_n$ the set of \defn{$\f{X}$-decorated permutations} of size~$n$, \ie of pairs~$(\sigma, \c{X})$ with~$\sigma \in \fS_n$ and~$\c{X} \in \f{X}_n$.
We consider the graded set~$\eP \eqdef \bigsqcup_{n \ge 0} \eP_n$ and the graded vector space~$\b{k}\eP \eqdef \bigoplus_{n\ge 0} \b{k}\eP_n$, where~$\b{k}\eP_n$ is a vector space with basis~$(\F_{(\sigma, \c{X})})_{(\sigma, \c{X}) \in \eP_n}$ indexed by \mbox{$\f{X}$-decorated} permutations of size~$n$.
For two decorated permutations~$(\sigma, \c{X})$ and~$(\tau, \c{Y})$, we define the \defn{product}~$\F_{(\sigma, \c{X})} \product \F_{(\tau, \c{Y})}$ by
\[
\F_{(\sigma, \c{X})} \product \F_{(\tau, \c{Y})} \eqdef \sum_{\rho \in \sigma \shiftedShuffle \tau} \F_{( \rho, \concat{\c{X}}{\c{Y}} )}.
\]

\begin{proposition}
\label{prop:algebraDecoratedPermutations}
The product~$\product$ defines an associative graded algebra structure on~$\b{k}\eP$.
\end{proposition}

\begin{proof}
If~$(\sigma, \c{X}) \in \eP_m$ and~$(\tau, \c{Y}) \in \eP_n$, we have~$(\rho, \concat{\c{X}}{\c{Y}}) \in \eP_{m+n}$ for any~$\rho \in \sigma \shiftedShuffle \tau$, so that~$\product$ is a graded product.
It is associative since both the concatenation (by Definition~\ref{def:decorationSet}\,(i)) and the shifted shuffle product (by Theorem~\ref{thm:MalvenutoReutenauer}) are associative.
\end{proof}

Note that an analogue of Proposition~\ref{prop:productIntervalWeakOrder} clearly holds for decorated permutations.

\medskip
We define the \defn{standardization} of a decorated permutation~$(\rho, \c{Z}) \in \eP_p$ at a subset~$R \subseteq [p]$ as
\[
\std{(\rho, \c{Z})}{R} \eqdef \big( \stdpos{\rho}{R}, \select{\c{Z}}{\rho^{-1}(R)} \big),
\]
where~$\stdpos{\rho}{R}$ is the position standardization on permutations and~$\select{\c{Z}}{\rho^{-1}(R)}$ is the selection on~$\f{X}$.
For a decorated permutation~$(\rho, \c{Z}) \in \eP_p$, we define the \defn{coproduct}~$\coproduct\F_{(\rho, \c{Z})}$ by
\[
\coproduct\F_{(\rho, \c{Z})} \eqdef \sum_{k = 0}^p \F_{\std{(\rho, \c{Z})}{[k]}} \otimes \F_{\std{(\rho, \c{Z})}{[p] \ssm [k]}}.
\]

\begin{proposition}
\label{prop:coalgebraDecoratedPermutations}
The coproduct~$\coproduct$ defines a coassociative graded coalgebra structure on~$\b{k}\eP$.
\end{proposition}

\begin{proof}
If~$(\sigma, \c{Z}) \in \eP_p$ and~$R \subseteq [p]$, we have~$\std{(\rho, \c{Z})}{R} \in \eP_{|R|}$ so that~$\coproduct$ is a graded coproduct. 
Moreover, using Definition~\ref{def:decorationSet}\,(ii), we obtain that for a decorated permutation~$(\rho, \c{Z}) \in \eP_p$, both~$(\Id \otimes \coproduct)(\coproduct\F_{(\rho,\c{Z})})$ and~$(\coproduct \otimes \Id)(\coproduct\F_{(\rho,\c{Z})})$ equal to
\[
\sum_{0 \le k \le \ell \le p} \F_{\std{(\rho, \c{Z})}{[k]}} \otimes \F_{\std{(\rho, \c{Z})}{[\ell] \ssm [k]}} \otimes \F_{\std{(\rho, \c{Z})}{[p] \ssm [\ell]}}. \qedhere
\]
\end{proof}

\begin{theorem}
\label{thm:HopfAlgebraDecoratedPermutations}
The product~$\product$ and coproduct~$\coproduct$ endow the family of decorated permutations with a graded Hopf algebra structure.
\end{theorem}

\begin{proof}
This follows from the same property for the product and coproduct on permutations and by Definition~\ref{def:decorationSet}\,(iii). Indeed
\begin{align*}
\coproduct\big( \F_{(\sigma, \c{X})} \product \F_{(\tau, \c{Y})} \big)
& = \sum_{k = 0}^{m+n} \sum_{\rho \in \sigma \shiftedShuffle \tau} \F_{\std{(\rho, \concat{\c{X}}{\c{Y}})}{[k]}} \otimes \F_{\std{(\rho, \concat{\c{X}}{\c{Y}})}{[m+n] \ssm [k]}} \\
& =  \sum_{p = 0}^m \sum_{q = 0}^n \begin{array}{@{}l} \Big( \sum\limits_\mu \F_{( \mu, \concat{\std{\c{X}}{\sigma^{-1}([p])}}{\std{\c{Y}}{\tau^{-1}([q])}} )} \Big) \\[-.1cm] \qquad \otimes \Big( \sum\limits_\nu \F_{( \nu, \concat{\std{\c{X}}{\sigma^{-1}([m] \ssm [p])}}{\std{\c{Y}}{\tau^{-1}([n] \ssm [q])}} )} \Big) \end{array} \\[-.1cm]
& = \Big( \sum_{p = 0}^m \F_{\std{(\sigma, \c{X})}{[p]}} \otimes \F_{\std{(\sigma, \c{X})}{[m] \ssm [p]}} \Big) \product \Big( \sum_{q = 0}^n \F_{\std{(\tau, \c{X})}{[q]}} \otimes \F_{\std{(\tau, \c{X})}{[n] \ssm [q]}} \Big) \\
& = \coproduct\F_{(\sigma, \c{X})} \product \coproduct\F_{(\tau, \c{Y})}
\end{align*}
where~$\mu$ ranges over~$\stdpos{\sigma}{[p]} \shiftedShuffle \stdpos{\tau}{[q]}$ while~$\nu$ ranges over~$\stdpos{\sigma}{[m] \ssm [p]} \shiftedShuffle \stdpos{\tau}{[n] \ssm [q]}$ in the second line, and the swap is understood in the last two lines.
\end{proof}

\begin{example}
When~$\f{X}$ is the set of words~$\c{A}^*$ on a finite alphabet~$\c{A}$ (with the classical concatenation of words and the selection defined by subwords, as in Example~\ref{exm:decorationSet}), the Hopf algebra of decorated permutations was studied in detail by J.-C.~Novelli and J.-Y.~Thibon in~\cite{NovelliThibon-coloredHopfAlgebras}.
Further relevant examples will appear in Section~\ref{sec:examples}.
\end{example}

%%%%%%%

\subsection{Decorated noncrossing arc diagrams}

We now use our Hopf algebra on decorated permutations to construct Hopf algebras on decorated noncrossing arc diagrams.
As in the previous section, we consider a decoration set~$(\f{X}, \concatf, \selectf)$ and the corresponding Hopf algebra~$(\b{k}\eP, \,\product\,, \coproduct)$ on $\f{X}$-decorated permutations.
Recall from Section~\ref{subsec:arcIdeals} that~$\f{I}_n$ denotes the set of arc ideals of~$\c{A}_n$.

For an arc~$\alpha = (a, b, m, S)$ and~$n \in \N$, we define the \defn{augmented arc}~$\alpha^{+n} \eqdef (a, b, m+n, S)$ and the \defn{shifted arc}~$\alpha^{\rightarrow n} \eqdef (a+n, b+n, m+n, \set{s+n}{s \in S})$.
Graphically, $\alpha^{+n}$ is obtained from~$\alpha$ by adding $n$ points to its right, and~$\alpha^{\rightarrow n}$ is obtained from~$\alpha$ by adding~$n$ points to its left.
For~$\c{I} \subseteq \c{A}_m$ and~$n \in \N$, define~$\c{I}^{+n} \eqdef \set{\alpha^{+n}}{\alpha \in \c{I}}$ and~$\c{I}^{\rightarrow n} \eqdef \set{\alpha^{\rightarrow n}}{\alpha \in \c{I}}$.

\begin{definition}
\label{def:conservative}
A graded function~$\arcs : \f{X} = \bigsqcup_{n \ge 0} \f{X}_n \longrightarrow \f{I} = \bigsqcup_{n \ge 0} \f{I}_n$ is \defn{conservative} if
\begin{enumerate}[(i)]
\item $\arcs(\c{X})^{+n}$ and~$\arcs(\c{Y})^{\rightarrow m}$ are both subsets of~$\arcs(\concat{\c{X}}{\c{Y}})$ for any~$\c{X} \in \f{X}_m$ and~${\c{Y} \in \f{X}_n}$,
\item $(r_a, r_b, p, S) \in \arcs(\c{Z})$ implies $(a, b, q, \set{c \in [q]}{r_c \in S}) \in \arcs(\select{\c{Z}}{R})$ for any~$\c{Z} \in \f{X}_p$, any ${R = \{r_1 < \dots < r_q\} \subseteq [p]}$, any~$1 \le a < b \le q$ and any~$S \subseteq {]r_a, r_b[}$.
\end{enumerate}
\end{definition}

\begin{example}
If~$\f{X} = \{\bullet\}^*$ is the decoration set of words on a one element alphabet, then the maps $\bullet^n \mapsto \c{A}_n = \set{(a, b, n, S)}{1 \le a < b \le n \text{ and } S \subseteq {]a,b[}}$ and~$\bullet^n \mapsto \set{(a, b, n, \varnothing)}{1 \le a < b \le n}$ are both conservative.
Further relevant examples will appear in Section~\ref{sec:examples}.
\end{example}

From now on, we assume that we are given a conservative function~$\arcs : \f{X} \longrightarrow \f{I}$.
For~${n \ge 0}$, we denote by~$\eD_n$ the set of \defn{$\f{X}$-decorated noncrossing arc diagrams} of size~$n$, \ie of pairs~$(\c{D}, \c{X})$ where~${\c{X} \in \f{X}_n}$ and~$\c{D}$ is a noncrossing arc diagram contained in~$\arcs(\c{X})$.

We now define a Hopf algebra on $\f{X}$-decorated noncrossing arc diagrams using the map~$\underline{\eta}$ defined in Section~\ref{subsec:surjection}.
We denote by~$\b{k}\eD \eqdef \bigoplus_{n \ge 0} \b{k}\eD_n$ the graded vector subspace of~$\b{k}\eP$ generated by the elements
\[
\PPT_{(\c{D}, \c{X})} \eqdef \sum_{\substack{\sigma \in \fS \\ \underline{\eta}{}_{\arcs(\c{X})}(\sigma) = \c{D}}} \F_{(\sigma, \c{X})},
\]
for all $\f{X}$-decorated noncrossing arc diagrams~$(\c{D}, \c{X})$.
Our central result is the following statement.

\begin{theorem}
\label{thm:HopfAlgebraDecoratedArcDiagrams}
The subspace~$\b{k}\eD$ is a Hopf subalgebra~of~$\b{k}\eP$.
\end{theorem}

\begin{proof}
We first prove that~$\b{k}\eD$ is a subalgebra of~$\b{k}\eP$, \ie that it is stable by product.
Consider two $\f{X}$-decorated noncrossing arc diagrams~$(\c{D}, \c{X}) \in \f{D}_m$ and~$(\c{E}, \c{Y}) \in \f{D}_n$.
By definition, all permutations that appear in the product~$\PPT_{(\c{D}, \c{X})} \product \PPT_{(\c{E}, \c{Y})}$ are decorated by the product~$\c{Z} \eqdef \concat{\c{X}}{\c{Y}}$.
Observe first that the product~$\PPT_{(\c{D}, \c{X})} \product \PPT_{(\c{E}, \c{Y})}$ is multiplicity-free in the basis~$(\F_{(\rho, \c{Z})})_{(\rho, \c{Z}) \in \eP}$ of~$\b{k}\eP$, since any~$\F_{(\rho, \c{Z})}$ in~$\PPT_{(\c{D}, \c{X})} \product \PPT_{(\c{E}, \c{Y})}$ must come from the product~$\F_{(\stdval{\rho}{[m]}, \c{X})} \product \F_{(\stdval{\rho}{[m+n] \ssm [m]}, \c{Y})}$.
Consider now two decorated permutations~$(\rho, \c{Z})$ and~$(\rho', \c{Z})$ such that~$\underline{\eta}{}_{\arcs(\c{Z})}(\rho) = \underline{\eta}{}_{\arcs(\c{Z})}(\rho')$.
We want to show that~$\F_{(\rho, \c{Z})}$ appears in~$\PPT_{(\c{D}, \c{X})} \product \PPT_{(\c{E}, \c{Y})}$ if and only if~$\F_{(\rho', \c{Z})}$ appears in~$\PPT_{(\c{D}, \c{X})} \product \PPT_{(\c{E}, \c{Y})}$.
Assume first that~${\rho' = \rho \cdot (k \; k+1)}$ for a descent~$k$ of~$\rho$ such that~${\underline{\alpha}(k, k+1, \rho)} \notin \arcs(\c{Z})$ and that~$\F_{(\rho, \c{Z})}$ appears in the product~$\PPT_{(\c{D}, \c{X})} \product \PPT_{(\c{E}, \c{Y})}$.
Let~$\sigma \in \underline{\eta}{}_{\arcs(\c{X})}^{-1}(\c{D})$ and~$\tau \in \underline{\eta}{}_{\arcs(\c{Y})}^{-1}(\c{E})$ be such that~$\rho \in \sigma \shiftedShuffle \tau$.
We then distinguish three cases:
\begin{itemize}
\item If $\rho_{k+1} \le m < \rho_k$, then~$\rho'$ also belongs to~$\sigma \shiftedShuffle \tau$ so that~$\F_{(\rho', \c{Z})}$ also appears in~${\PPT_{(\c{D}, \c{X})} \product \PPT_{(\c{E}, \c{Y})}}$.
\item If~$\rho_k \le m$, then Definition~\ref{def:conservative}\,(i) ensures that~${\underline{\alpha}(i, i+1, \sigma) \notin \c{X}}$, where~$i$ is the descent of~$\sigma$ such that~$\sigma_i = \rho_k$. The permutation~$\sigma' \eqdef \sigma \cdot (i \; i+1)$ thus satisfies~${\underline{\eta}{}_{\arcs(\c{X})}(\sigma') = \underline{\eta}{}_{\arcs(\c{X})}(\sigma) = \c{D}}$. Since~$\rho' \in \sigma' \shiftedShuffle \tau$, it follows that~$\F_{(\rho', \c{Z})}$ also appears in~${\PPT_{(\c{D}, \c{X})} \product \PPT_{(\c{E}, \c{Y})}}$.
\item Finally, if~$m < \rho_{k+1}$, the argument is similar using~$\tau' \eqdef \tau \cdot (j \; j+1)$ where~$\tau_j = \rho_k$.
\end{itemize}
By transitivity in Theorem~\ref{thm:arcDiagrams}\,(iii), we obtain that~$\F_{(\rho, \c{Z})}$ appears in~$\PPT_{(\c{D}, \c{X})} \product \PPT_{(\c{E}, \c{Y})}$ if and only if~$\F_{(\rho', \c{Z})}$ appears in~$\PPT_{(\c{D}, \c{X})} \product \PPT_{(\c{E}, \c{Y})}$.
Therefore~$\b{k}\eD$ is a subalgebra of~$\b{k}\eP$.

We now prove that~$\b{k}\f{D}$ is a subcoalgebra of~$\b{k}\f{P}$, \ie that it is stable by coproduct.
Consider a decorated noncrossing arc diagram~$(\c{D}, \c{Z}) \in \f{D}_p$.
Observe first that the coproduct~$\coproduct\PPT_{(\c{D}, \c{Z})}$ is multiplicity-free in the basis~$(\F_{(\rho, \c{Z})})_{(\rho, \c{Z}) \in \eP}$ of~$\b{k}\eP$, since any~${\F_{(\sigma, \c{X})} \otimes \F_{(\tau, \c{Y})}}$ in~$\coproduct\PPT_{(\c{D}, \c{Z})}$ must come from the coproduct~$\coproduct\F_{(\rho, \c{Z})}$ where~$\rho(i) = \sigma(i)$ if~$i \in [m]$ and~$\rho(i) = m + \sigma(i-m)$ if~${i \in [m+n] \ssm [m]}$.
Consider now four decorated permutations~$(\sigma, \c{X})$, $(\sigma', \c{X})$, $(\tau, \c{Y})$ and~$(\tau', \c{Y})$ such that ${\underline{\eta}{}_{\arcs(\c{X})}(\sigma) = \underline{\eta}{}_{\arcs(\c{X})}(\sigma')}$ and~$\underline{\eta}{}_{\arcs(\c{Y})}(\tau) = \underline{\eta}{}_{\arcs(\c{Y})}(\tau')$.
We want to show that~${\F_{(\sigma, \c{X})} \otimes \F_{(\tau, \c{Y})}}$ appears in the coproduct~$\coproduct\PPT_{(\c{D}, \c{Z})}$ if and only if~${\F_{(\sigma', \c{X})} \otimes \F_{(\tau', \c{Y})}}$ appears in the coproduct~$\coproduct\PPT_{(\c{D}, \c{Z})}$.
Assume first that~$\sigma' = \sigma \cdot (i \; i+1)$ for a descent~$i$ of~$\sigma$ with~${\underline{\alpha}(i, i+1, \sigma) \notin \c{X}}$, that~$\tau = \tau'$ and that~$\F_{(\sigma, \c{X})} \otimes \F_{(\tau, \c{Y})}$ appears in the coproduct~$\coproduct\PPT_{(\c{D}, \c{Z})}$.
By definition of the coproduct, there exist~${\rho \in \underline{\eta}{}_{\arcs(\c{Z})}^{-1}(\c{D})}$ and~$k \in [p]$ such that~${(\sigma, \c{X}) = \std{(\rho, \c{Z})}{[k]}}$ and~${(\tau, \c{Y}) = \std{(\rho, \c{Z})}{[p] \ssm [k]}}$.
Since~${\c{X} = \select{\c{Z}}{\rho^{-1}([k])}}$ and~$\underline{\alpha}(i, i+1, \sigma) \notin \c{X}$, Definition~\ref{def:conservative}\,(ii) ensures that~$\underline{\alpha}(i, i+1, \rho) \notin \c{Z}$.
We conclude that~$\rho' = \rho \cdot (i \; i+1)$ is in~$\underline{\eta}{}_{\arcs(\c{Z})}^{-1}(\c{D})$.
Since moreover~${(\sigma', \c{X}) = \std{(\rho', \c{Z})}{[k]}}$ and~${(\tau, \c{Y}) = \std{(\rho', \c{Z})}{[p] \ssm [k]}}$, we get that~${\F_{(\sigma', \c{X})} \otimes \F_{(\tau, \c{Y})}}$ appears in the coproduct~$\coproduct\PPT_{(\c{D}, \c{Z})}$.
By symmetry and by transitivity in Theorem~\ref{thm:arcDiagrams}\,(ii), we conclude that~$\F_{(\sigma, \c{X})} \otimes \F_{(\tau, \c{Y})}$ appears in the coproduct~$\coproduct\PPT_{(\c{D}, \c{Z})}$ if and only if~$\F_{(\sigma', \c{X})} \otimes \F_{(\tau', \c{Y})}$ appears in the coproduct~$\coproduct\PPT_{(\c{D}, \c{Z})}$ for any~$(\sigma, \c{X})$, $(\sigma', \c{X})$, $(\tau, \c{Y})$ and~$(\tau', \c{Y})$ such that~$\underline{\eta}{}_{\arcs(\c{X})}(\sigma) = \underline{\eta}{}_{\arcs(\c{X})}(\sigma')$ and~$\underline{\eta}{}_{\arcs(\c{Y})}(\tau) = \underline{\eta}{}_{\arcs(\c{Y})}(\tau')$.
Therefore~$\b{k}\eD$ is a subcoalgebra of~$\b{k}\eP$.
\end{proof}

%We now provide explicit descriptions of the product and the coproduct in~$\b{k}\eD$.

\enlargethispage{.7cm}
We now provide an analogue of Proposition~\ref{prop:productIntervalWeakOrder} for decorated noncrossing arc diagrams.

\begin{proposition}
\label{prop:descriptionProduct}
Consider two $\f{X}$-decorated noncrossing arc diagrams~$(\c{D}, \c{X})$ and~$(\c{E}, \c{Y})$, and their corresponding weak order intervals~$[\mu, \nu] \eqdef \underline{\eta}_{\arcs(\c{X})}^{-1}(\c{D})$ and~$[\lambda, \omega]  \eqdef \underline{\eta}_{\arcs(\c{Y})}^{-1}(\c{E})$.
Then
\[
\PPT_{(\c{D}, \c{X})} \product \PPT_{(\c{E}, \c{Y})} = \sum_{\c{F}} \PPT_{(\c{F}, \concat{\c{X}}{\c{Y}} )},
\]
where~$\c{F}$ ranges in the interval between~$\c{D}\backslash\c{E} \eqdef \underline{\eta}_{\arcs(\concat{\c{X}}{\c{Y}})}(\mu\backslash\lambda)$ and~$\c{E}\slash\c{D} \eqdef \underline{\eta}_{\arcs(\concat{\c{X}}{\c{Y}})}(\omega\slash\nu)$ in the lattice of noncrossing arc diagrams in~$\arcs(\concat{\c{X}}{\c{Y}})$.
\end{proposition}

\begin{proof}
This follows from Proposition~\ref{prop:productIntervalWeakOrder} and the fact that the $\underline{\eta}_{\arcs(\concat{\c{X}}{\c{Y}})}$ fibers of the noncrossing arc diagrams in the $\arcs(\concat{\c{X}}{\c{Y}})$-interval between~$\c{D}\backslash\c{E}$ and~$\c{E}\slash\c{D}$ form a partition of the  weak order interval between~$\mu\backslash\lambda$ and~$\omega\slash\nu$:
\[
\PPT_{(\c{D}, \c{X})} \product \PPT_{(\c{E}, \c{Y})}
= \sum_{\mu \le \sigma \le \nu} \!\! \F_{(\sigma, \c{X})} \product \!\! \sum_{\lambda \le \tau \le \omega} \!\! \F_{(\tau, \c{Y})}
= \!\!\!\! \sum_{\mu\backslash\lambda \le \rho \le \omega\slash\nu} \!\!\!\! \F_{(\rho, \concat{\c{X}}{\c{Y}})}
= \sum_{\c{F}} \PPT_{(\c{F}, \concat{\c{X}}{\c{Y}} )}
\qedhere
\]
\end{proof}

%In contrast, the coproduct is more difficult to describe explicitly on noncrossing arc diagrams. It can be done by considering cuts in the noncrossing arc bidiagrams.
%
%\begin{proposition}
%\label{prop:descriptionCoproduct}
%\end{proposition}

\subsection{Decoration products and decoration subsets}

To close our generic Hopf algebra constructions, we observe two natural operations on decoration sets that behave properly with our construction.
The straightforward proofs are left to the reader.
First, we observe that we can obtain Hopf algebra structures on arc diagrams decorated by cartesian products of decorations.

\begin{proposition}
\label{prop:decorationProduct}
The Cartesian product~$\f{X} \times \f{X}'$ of two decoration sets~$\f{X}$ and~$\f{X}'$, endowed with the concatenation and selection defined by
\begin{align*}
\concat{(\c{X},\c{X}')}{(\c{Y},\c{Y}')} & \eqdef (\concat{\c{X}}{\c{Y}}, \concat{\c{X}'}{\c{Y}'}) \\
\text{and} \qquad \select{(\c{Z},\c{Z}')}{R} & \eqdef (\select{\c{Z}}{R}, \select{\c{Z}'}{R}),
\end{align*}
is a decoration set.
Moreover, for two conservative functions~$\arcs: \f{X} \to \f{I}$ and~$\arcs': \f{X}' \to \f{I}$, the function~$\arcs \cap \arcs': \f{X} \times \f{X}' \to \f{I}$ defined by~$(\arcs \cap \arcs') (\c{X}, \c{X}') \eqdef \arcs(\c{X}) \cap \arcs'(\c{X}')$ is conservative.
\end{proposition}

On the other hand, we observe that decoration subsets define Hopf subalgebras.

\begin{proposition}
\label{prop:decorationSubset}
If~$\f{X} \subseteq \f{X'}$ are two decoration sets, and~$\arcs': \f{X}' \to \f{I}$ is a conservative function, then the restriction~$\arcs$ of~$\arcs'$ to~$\f{X}$ is conservative.
Therefore, the Hopf algebra~$\b{k}\f{D}$ constructed from~$\arcs$ is a Hopf subalgebra of the Hopf algebra~$\b{k}\f{D}'$ constructed from~$\arcs'$.
\end{proposition}

%%%%%%%%%%%%%%%%%%%%%%%%%%%%%%%%%%%%%%

\section{Applications}
\label{sec:examples}

In this section, we provide examples of applications of Theorems~\ref{thm:HopfAlgebraDecoratedPermutations} and~\ref{thm:HopfAlgebraDecoratedArcDiagrams}.

%%%%%%%

\subsection{Insertional, translational, and Hopf families of congruences}

For all~$n \in \N$, fix a lattice congruence~$\equiv_n$ of the weak order on~$\f{S}_n$, with arc ideal~$\c{I}_n$.
As a first application of Theorem~\ref{thm:HopfAlgebraDecoratedArcDiagrams}, we obtain sufficient conditions for the family~$(\equiv_n)_{n \in \N}$ to define a Hopf subalgebra~of~$\b{k}\f{S}$.
%Define the \defn{shifted concatenation} of~$\sigma = \sigma_1 \dots \sigma_m \in \f{S}_m$ and~$\tau = \tau_1 \dots \tau_n \in \f{S}_n$ by~${\sigma\overline{\tau} \eqdef \sigma_1 \dots \sigma_m (\tau_1 + m) \dots (\tau_n + m) \in \f{S}_{m+n}}$.

\begin{corollary}[{\cite[Thm.~1.2 \& 1.3]{Reading-HopfAlgebras}}]
\label{coro:HopfFamilies}
For all~$n \in \N$, consider a lattice congruence~$\equiv_n$ of the weak order on~$\f{S}_n$, with arc ideal~$\c{I}_n$.
If
\begin{itemize}
\item both~$\c{I}_m^{+n}$ and~$\c{I}_n^{\rightarrow m}$ are contained in~$\c{I}_{m+n}$ for all~$m, n \in \N$,
\item $(r_a, r_b, p, S) \in \c{I}_p$ implies~$(a, b, q, \set{c \in [q]}{r_c \in S}) \in \c{I}_q$ for any~${R = \{r_1 < \dots < r_q\} \subseteq [p]}$, any~$1 \le a < b \le q$ and any~$S \subseteq {]r_a, r_b[}$,
\end{itemize}
then the subvector space of~$\b{k}\f{S}$ generated by the sums~$\sum_{\sigma} \F_\sigma$ over the classes of the congruences~$\equiv_n$ is a Hopf subalgebra of C.~Malvenuto and C.~Reutenauer's Hopf algebra~$\b{k}\f{S}$.
\end{corollary}

\begin{proof}
Consider the decoration set~$\{\bullet\}^*$ of words over a one element alphabet and the function ${\arcs : \{\bullet\}^* \to \f{I}}$ given by~${\arcs(\bullet^n) = \c{I}_{\equiv_n}}$.
Note that the Hopf algebra~$\b{k}\f{P}$ of permutations decorated with~$\{\bullet\}^*$ is just isomorphic to C.~Malvenuto and C.~Reutenauer's Hopf algebra~$\b{k}\f{S}$ on permutations.
Moreover, the conditions of the statement assert that~$\arcs$ is conservative.
The result thus immediately follows from Theorem~\ref{thm:HopfAlgebraDecoratedArcDiagrams}.
\end{proof}

The conditions of Corollary~\ref{coro:HopfFamilies} are essentially the translational and insertional conditions given by N.~Reading in~\cite[Thm.~1.2 \& 1.3]{Reading-HopfAlgebras}.
Note however that our condition is slightly weaker as we only require that the restriction of~$\equiv_{m+n}$ to~$\set{\sigma\backslash\tau}{\sigma \in \f{S}_m \text{ and } \tau \in \f{S}_n}$ refines the congruence relation induced by~${\equiv_m} \times {\equiv_n}$ on~$\set{\sigma\backslash\tau}{\sigma \in \f{S}_m \text{ and } \tau \in \f{S}_n}$ while N.~Reading's definition of translational families requires that these congruences coincide.

\begin{example}
Corollary~\ref{coro:HopfFamilies} covers various families of lattice congruences, producing Hopf algebra structures on permutations~\cite{MalvenutoReutenauer}, on binary trees~\cite{LodayRonco, HivertNovelliThibon-algebraBinarySearchTrees}, on binary sequences~\cite{GelfandKrobLascouxLeclercRetakhThibon}, on diagonal quadragulations~\cite{LawReading} or equivalently on twin binary trees~\cite{Giraudo}, on $k$-twists~\cite{Pilaud-brickAlgebra}, on $k$-descent schemes~\cite{NovelliReutenauerThibon, Pilaud-brickAlgebra}, etc.
\end{example}

%%%%%%%

\subsection{Bounded crossings}
\label{subsec:boundedCrossings}

We now consider the family of arc ideals~$\c{A}_{\north, \south, \east, \west}^{< k}$ defined in Example~\ref{exm:boundedCrossingCongruences}.
Consider the decoration set~$\f{X} = \bigsqcup_{n \in \N} \f{X}_n$, where~$\f{X}_n$ is the set of quadruples of functions~$[n] \to \N$, and where the concatenation is defined by
\[
\concat{(\north, \south, \east, \west)}{(\north', \south', \east', \west')}(a) =
\begin{cases}
(\north(a), \south(a), \east(a), \west(a)) & \text{if } a \le m \\
(\north'(a), \south'(a), \east'(a), \west'(a)) & \text{if } a > m
\end{cases}
\]
(in other words, the usual concatenation of words in~$(\N^4)^*$), and the selection is defined by
\[
\select{(\north, \south, \east, \west)}{R}(a) = (\north(r_a), \; \south(r_a), \; \min \set{\east(s)}{s \in {]r_{a-1}, r_a]}}, \; \min \set{\west(s)}{s \in {[r_a, r_{a+1}[}}).
\]
Choose~$k \in \N$ and define the function~${\arcs : \f{X}_n \to \c{A}_n}$ by~$\arcs(\north, \south, \east, \west) = \c{A}_{\north, \south, \east, \west}^{< k}$.
Recall from Example~\ref{exm:boundedCrossingCongruences} that for each~$a \in [n]$, we place $\north(a)$ upper vertical walls above~$a$, $\south(a)$ lower vertical walls below~$a$ and~$\min(\east(a), \west(a+1))$ horizontal walls between~$a$ and~$a+1$, and that an arc belongs to~$\c{A}_{\north, \south, \east, \west}^{< k}$ if it crosses at most~$k-1$ of these walls.
The function~$\arcs$ is conservative since
\begin{itemize}
\item for any~$u \in \f{X}_m$, $v \in \f{X}_n$, $\alpha \in \arcs(u)$ and~$\beta \in \arcs(v)$, the walls of~$uv$ crossed by~$\alpha^{+n}$ are precisely the walls of~$u$ crossed by~$\alpha$, while the walls of~$uv$ crossed by~$\beta^{\rightarrow m}$ are precisely the $m$-translates of the walls of~$v$ crossed by~$\beta$,
\item for any~$w \in \f{X}_p$, $R = \{r_1, \dots, r_q\} \subseteq [p]$, $1 \le a < b \le q$ and~$S \subseteq {]r_a, r_b[}$, the walls crossed by the arc~$(a, b, q, \set{c}{r_c \in S})$ are walls crossed by the arc~$(r_a, r_b, p, S)$ (but the latter might cross more walls than the former).
\end{itemize}
We therefore obtain a Hopf algebra~$\b{k}\f{D}^{< k}$ on the classes of all lattice congruences~$\c{A}_{\north, \south, \east, \west}^{< k}$ simultaneously.
Moreover, as observed in Proposition~\ref{prop:decorationSubset} any subset of~$\f{X}$ stable by concatenation and selection provides a Hopf subalgebra~$\b{k}\f{D}^{< k}$.
In particular, $\b{k}\f{D}^{< 1}$ contains \defn{simultaneously} Hopf subalgebras on permutations~\cite{MalvenutoReutenauer}, binary trees~\cite{LodayRonco, HivertNovelliThibon-algebraBinarySearchTrees}, binary sequences~\cite{GelfandKrobLascouxLeclercRetakhThibon}, Cambrian trees~\cite{ChatelPilaud}, permutrees~\cite{PilaudPons-permutrees}, and diagonal rectangulations~\cite{LawReading, Giraudo}, while~$\b{k}\f{D}^{< k}$ contains subalgebras on $k$-twists~\cite{Pilaud-brickAlgebra} and on $k$-descent schemes~\cite{NovelliReutenauerThibon, Pilaud-brickAlgebra}.
Finally, one could also mix conditions on the crossing numbers with walls of different colors using \mbox{Proposition~\ref{prop:decorationProduct}}.

%%%%%%%

\subsection{All arc diagrams}

To conclude, we define a Hopf algebra~$\b{k}\f{D}^\star$ simultaneously involving the classes of all lattice congruences of the weak order, and containing the permutree algebra.

\subsubsection{Extended arcs}

We call \defn{extended arc} any quadruple~$(a, b, n, S)$ where~$a,b,n \in \N$ are such that~$0 \le a < b \le n+1$, and~${S \subseteq {]a,b[}}$.
In other words, extended arcs are precisely like arcs but they are allowed to be attached before~$1$ and after~$n$.
We represent extended arcs exactly in the same way as arcs (but the points $0$ and~$n+1$ are colored white).
We denote by~$\c{A}_n^\star \eqdef \set{(a, b, n, S)}{0 \le a < b \le n+1 \text{ and } S \subseteq {]a,b[}}$ the set of all extended arcs.
An extended arc~$(a, b, n, S)$ is \defn{initial} if~$a = 0$, \defn{terminal} if~$b = n+1$, and \defn{strict} otherwise.
The notions of crossing and forcing, as well as the operations~$\alpha^{+n}$ and~$\alpha^{\rightarrow n}$, are defined as for classical~arcs.
We denote by~$\f{I}_n^\star$ the set of extended arc ideals (\ie upper ideals of the forcing order~$\prec$ on~$\c{A}_n^\star$).

We call \defn{juxtaposition}~$\juxta{\alpha}{\beta}$ of two extended arcs~$\alpha \eqdef (a, b, n, R)$ and~$\beta \eqdef (c, d, n, S)$ the arc
\[
\juxta{\alpha}{\beta} \eqdef
\begin{cases}
	\{(a, d, n, R \cup S)\} & \text{if~$c = b + 1$,} \\
	\varnothing & \text{otherwise.}
\end{cases}
\]
In other words, $\alpha\beta$ is obtained by joining $\alpha$ to~$\beta$ when the final endpoint of~$\alpha$ appears just after the initial endpoint of~$\beta$, and is empty otherwise.
Note that the juxtaposition is associative: for $\alpha \eqdef (a, b, n, R)$, $\beta \eqdef (c, d, n, S)$, and~$\gamma \eqdef (e, f, n, T)$ we have~$\juxta{(\juxta{\alpha}{\beta})}{\gamma} = \juxta{\alpha}{(\juxta{\beta}{\gamma})} = \{(a, f, n, R \cup S \cup T)\}$ if~$c = b + 1$ and~$e = d + 1$, and~$\juxta{(\juxta{\alpha}{\beta})}{\gamma} = \juxta{\alpha}{(\juxta{\beta}{\gamma})} = \varnothing$ otherwise.
For~$\c{I}, \c{J} \subseteq \c{A}_n^\star$, we define the \defn{juxtaposition}~$\juxta{\c{I}}{\c{J}}$~by
\[
\juxta{\c{I}}{\c{J}} \eqdef \c{I} \cup \c{J} \cup \bigcup_{\substack{\alpha \in \c{I} \\ \beta \in \c{J}}} \juxta{\alpha}{\beta}
\]
Observe that if~$\c{I}$ and~$\c{J}$ are both extended arc ideals, then~$\juxta{\c{I}}{\c{J}}$ is also an extended arc ideal.
Note again that the juxtaposition is associative: $\juxta{(\juxta{\c{I}}{\c{J}})}{\c{K}} = \juxta{\c{I}}{(\juxta{\c{J}}{\c{K}})} = \c{I} \cup \c{J} \cup \bigcup_{\substack{\alpha \in \c{I} \\ \beta \in \c{J} \\ \gamma \in \c{K}}} \big( \juxta{\alpha}{\beta} \cup \juxta{\beta}{\gamma} \cup \juxta{\juxta{\alpha}{\beta}}{\gamma} \big)$.

We define the \defn{concatenation} of two extended arc ideals~$\c{I} \subseteq \c{A}_m^\star$ and~$\c{J} \subseteq \c{A}_n^\star$ by
\[
\concat{\c{I}}{\c{J}} \eqdef \juxta{\c{I}^{+n}}{\c{J}^{\rightarrow m}}.
\]
Graphically, $\concat{\c{I}}{\c{J}}$ is obtained by juxtaposing~$\c{I}$ and~$\c{J}$ such that the~$n$\,---\,$(n+1)$ edge of~$\c{I}$ coincides with the~$0$\,---\,$1$ edge of~$\c{J}$, and joining all final arcs of~$\c{I}$ with all initial arcs of~$\c{J}$.
See \fref{fig:arcConcatenation} for an illustration.

\begin{figure}[h]
	\capstart
	\centerline{\includegraphics[scale=1]{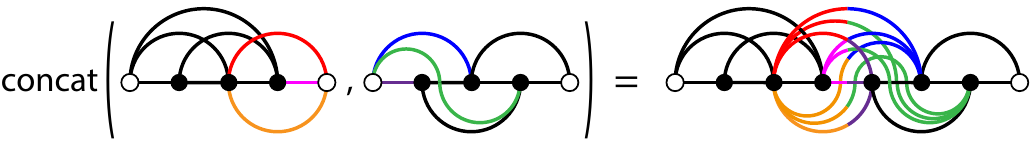}}
	\caption{The concatenation of two extended arc ideals.}
	\label{fig:arcConcatenation}
\end{figure}

\begin{lemma}
\label{lem:concatenation}
The concatenation~$\concat{\c{I}}{\c{J}}$ of two extended arc ideals~$\c{I} \subseteq \c{A}_m^\star$ and~$\c{J} \subseteq \c{A}_n^\star$ is an arc ideal of~$\c{A}_{n+m}^\star$.
\end{lemma}

\begin{proof}
The set~$\concat{\c{I}}{\c{J}}$ is contained in~$\c{A}_{m+n}^\star$ by definition and closed by forcing since both~$\c{I}^{+n}$ and~$\c{J}^{\rightarrow m}$~are.
\end{proof}

\begin{lemma}
\label{lem:concatenationAssociative}
For any arc ideals~$\c{I}, \c{J}, \c{K}$, we have $\concat{\concat{\c{I}}{\c{J}}}{\c{K}} = \concat{\c{I}}{\concat{\c{J}}{\c{K}}}$.
\end{lemma}

\begin{proof}
Assume that~$\c{I} \subseteq \c{A}_m^\star$, $\c{J} \subseteq \c{A}_n^\star$ and~$\c{K} \subseteq \c{A}_p^\star$.
Using the associativity of the juxtaposition,
\begin{align*}
\concat{\concat{\c{I}}{\c{J}}}{\c{K}}
& = (\c{I}^{+n}\c{J}^{\rightarrow m})^{+p}\c{K}^{\rightarrow (m+n)} = \c{I}^{+(n+p)}(\c{J}^{\rightarrow m})^{+p}\c{K}^{\rightarrow (m+n)} \\
& = \c{I}^{+(n+p)}(\c{J}^{+p}\c{K}^{\rightarrow n})^{\rightarrow m} = \concat{\c{I}}{\concat{\c{J}}{\c{K}}}. \qedhere
\end{align*}
\end{proof}

Consider now an extended arc ideal~$\c{K} \subseteq \c{A}_p^\star$ and a subset~$X \eqdef \{x_1 < \dots < x_q\}$ of~$[p]$. Define by convention~$x_0 \eqdef 0$ and~$x_{q+1} \eqdef p+1$.
We define the \defn{selection} of~$\c{K}$ at~$X$ by
\[
\select{\c{K}}{X} \eqdef \set{(a, b, q, S)}{\begin{array}{@{}l@{}} \exists \; y_0 < \dots < y_r \in [p] \text{ with } x_a = y_0 \text{ and } x_b = y_r \text{ while } y_1, \dots, y_{r-1} \notin X \\ \exists \; (y_0, y_1, p, S_1), \dots, (y_{r-1}, y_r, p, S_r) \in \c{K} \text{ with } S = \set{\ell \in [q]}{x_\ell \in \bigcup S_k}\end{array}}.
\]
Graphically, $\select{\c{K}}{X}$ is obtained by considering all arcs obtained by merging $x$-monotone paths in~$\c{K}$ with endpoints in~$\{0\} \cup X \cup \{p+1\}$ but all interior points in~$[p] \ssm X$, deleting all points of~$[p] \ssm X$, and packing the remaining points of~$\{0\} \cup X \cup \{p+1\}$ (together with the merged arcs) to the left towards their standard position~$0, 1, \dots, q, q+1$.
See \fref{fig:arcSelection} for an illustration.

\begin{figure}[h]
	\capstart
	\centerline{\includegraphics[scale=1]{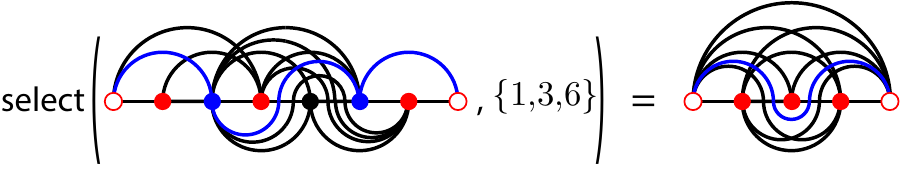}}
	\caption{The selection in an arc ideal. Selected points are in red. The blue extended arc~$(0,4,3,\{1,3\})$ in~$\select{\c{K}}{\{1,3,6\}}$ arises from the concatenation of the three blue extended arcs~$(0,2,6,\{1\})$, $(2,5,6,\{4\})$ and~$(5,7,6,\{6\})$ in~$\c{K}$.}
	\label{fig:arcSelection}
\end{figure}

\begin{lemma}
\label{lem:selection}
The selection~$\select{\c{K}}{X}$ of an extended arc ideal~$\c{K} \subseteq \c{A}_p^\star$ on a $q$-element subset~$X$ of~$[p]$ is an extended arc ideal of~$\c{A}_q^\star$.
\end{lemma}

\begin{proof}
Let~$\c{I} \eqdef \select{\c{K}}{X}$.
We have~$\c{I} \subseteq \c{A}_q^\star$ by definition.
To show that~$\c{I}$ is closed by forcing, assume that an arc~$(a, d, q, S) \in \c{I}$ is forced by an arc~$(b, c, q, T)$, so that~${a \le b < c \le d}$ and~${T = S \cap {]b, c[}}$.
Consider a path of arcs~$\alpha_1 \eqdef (y_0, y_1, p, S_1), \dots, \alpha_r \eqdef (y_{r-1}, y_r, p, S_r)$ of~$\c{K}$ corresponding to~$\c{I}$.
Let~$u \eqdef \min\set{s \in [r]}{x_b < y_s}$ and~$v \eqdef \max\set{s \in [r]}{y_{s-1} < x_c}$.
If~$u = v$, then~$(x_b, x_c, p, S_u \cap {]x_b, x_c[})$ belongs to~$\c{K}$ (since it is closed by forcing), thus~$(b, c, q, T)$ belongs to~$\c{I}$.
If~$u < v$, then both~$\alpha_u' \eqdef (x_b, y_u, p, S_u \cap {]x_b, y_u[})$ and ${\alpha_v' \eqdef (y_{v-1}, x_c, S_v \cap {]y_{v-1}, x_c[})}$ belong to~$\c{K}$ (since it is closed by forcing).
The path of arcs~$\alpha_u', \alpha_{u+1}, \dots, \alpha_{v-1}, \alpha_v'$ thus ensures that~$(b, c, q, T)$ belongs to~$\c{I}$ as well.
\end{proof}

\begin{lemma}
\label{lem:selectionCoassociative}
For any extended arc ideal~$\c{K} \subseteq \c{A}_p^\star$, any subset~$X \eqdef \{x_1, \dots, x_q\}$ of~$[p]$ and any subset~$Y$ of~$[q]$, we have~$\select{\select{\c{K}}{X}}{Y} = \select{\c{K}}{\set{x_y}{y \in Y}}$.
\end{lemma}

\begin{proof}
By definition, both~$\select{\select{\c{K}}{X}}{Y}$ and~$\select{\c{K}}{\set{x_y}{y \in Y}}$ are obtained by merging all paths in~$\c{K}$ whose endpoints are in~$\{0\} \cup \set{x_y}{y \in Y} \cup \{p+1\}$ but whose interior vertices are all in~$[p] \ssm \set{x_y}{y \in Y}$, and packing the remaining points of~$\{0\} \cup \set{x_y}{y \in Y} \cup \{p+1\}$ to their standard position~$0, 1, \dots, |Y|, |Y|+1$.
\end{proof}

\begin{proposition}
\label{prop:concatenationStandardization}
For any arc ideals~$\c{I} \subseteq \c{A}_m^\star$ and~$\c{J} \subseteq \c{A}_n^\star$ and any subsets~$X \subseteq [m]$ and~$Y \subseteq [n]$,
\[
\concat{\select{\c{I}}{X}}{\select{\c{J}}{Y}} = \select{\concat{\c{I}}{\c{J}}}{X \cup Y^{\rightarrow m}}
\]
where~$Y^{\rightarrow m} \eqdef \set{y+m}{y \in Y}$.
\end{proposition}

\begin{proof}
Consider an arc~$\alpha \eqdef (a, b, m+n, S) \in \c{A}_{m+n}^\star$.
We distinguish three cases:
\begin{itemize}
\item If~$b \le m$, then~$\alpha \in \concat{\select{\c{I}}{X}}{\select{\c{J}}{Y}}$ and~${\alpha \in \select{\concat{\c{I}}{\c{J}}}{X \cup Y^{\rightarrow m}}}$ are both equivalent to~$(a, b, m, S) \in \select{\c{I}}{X}$.
\item If~$m < a$, then~$\alpha \in \concat{\select{\c{I}}{X}}{\select{\c{J}}{Y}}$  and~${\alpha \in \select{\concat{\c{I}}{\c{J}}}{X \cup Y^{\rightarrow m}}}$ are both equivalent to~$(a-m, b-m, n, \set{s-m}{s \in S}) \in \select{\c{J}}{Y}$.
\item Finally, assume that~$a \le m < b$. If~$\alpha \in \concat{\select{\c{I}}{X}}{\select{\c{J}}{Y}}$, then there exists a final arc~$\beta \in \select{\c{I}}{X}$ and an initial arc~$\gamma \in \select{\c{J}}{Y}$ such that~$\{\alpha\} = \juxta{\beta^{+n}}{\gamma^{\rightarrow m}}$. The arc~$\beta \in \select{\c{I}}{X}$ corresponds to a path of arcs~$\beta_0, \dots, \beta_r$ in~$\c{I}$ whose interior points all belong to~$[m] \ssm X$, and similarly the arc~$\gamma \in \select{\c{J}}{Y}$ corresponds to a path of arcs~$\gamma_0, \dots, \gamma_s$ in~$\c{J}$ whose interior points all belong to~$[n] \ssm Y$. The path $\beta_0^{+n}, \dots, \beta_{r-1}^{+n}, \juxta{\beta_r^{+n}}{\gamma_0^{\rightarrow m}}, \gamma_1^{\rightarrow m}, \dots, \gamma_s^{\rightarrow m}$ thus shows that ${\alpha \in \select{\concat{\c{I}}{\c{J}}}{X \cup Y^{\rightarrow m}}}$.
Conversely, if~$\alpha \in \select{\concat{\c{I}}{\c{J}}}{X \cup Y^{\rightarrow m}}$, it corresponds to a path of arcs~$\alpha_1, \dots, \alpha_r$ in~$\concat{\c{I}}{\c{J}}$ whose interior points all belong to~$X \cup Y^{\rightarrow m}$. Since~$a \le m < b$, there is~$s \in [r]$ such that~$\alpha_s = (c, d, m+n, T)$ with~$c \le m < d$. Let~$\beta_1, \dots, \beta_s$ in~$\c{I}$ and~$\gamma_s, \dots, \gamma_r$ in~$\c{J}$ be such that~$\alpha_s = \juxta{\beta_s^{+n}}{\gamma_s^{\rightarrow m}}$ and~$\alpha_i = \beta_i^{+n}$ if~$i < s$ while $\alpha_i = \gamma_i^{\rightarrow m}$ if~$i > s$. Then the paths~$\beta_1, \dots, \beta_s$ in~$\c{I}$ and~$\gamma_s, \dots, \gamma_r$ in~$\c{J}$ ensure the existence of~$\beta \in \select{\c{I}}{X}$ and~$\gamma \in \select{\c{J}}{Y}$ such that~$\alpha = \juxta{\beta^{+n}}{\gamma^{\rightarrow m}} \in \concat{\select{\c{I}}{X}}{\select{\c{J}}{Y}}$.
\qedhere
\end{itemize}
\end{proof}

We have therefore proved the following statement.

\begin{corollary}
The set~$\f{I}^\star \eqdef \bigsqcup_{n \in \N} \f{I}_n^\star$ of all extended arcs ideals, endowed with the concatenation~$\concatf$ and selection~$\selectf$, is a decoration set.
\end{corollary}

\subsubsection{Noncrossing extended arc diagrams}

We now consider the map~$\arcs: \f{I}^\star \to \f{I}$ which sends an extended arc ideal to the arc ideal of its strict arcs.
This function is clearly conservative so that we obtain a Hopf algebra~$\b{k}\f{D}^\star$ on pairs~$(\c{D}, \c{I})$, where~$\c{I}$ is any extended arc ideal and $\c{D}$ is a noncrossing arc diagram containing only strict arcs of~$\c{I}$.
Note that~$\b{k}\f{D}^\star$ is graded but not connected as there are two extended arc ideals in~$\c{A}_0^\star$: the pair~$\big(\varnothing, \{(0, 1, 0, \varnothing)\} \big)$ is the identity of~$\b{k}\f{D}^\star$, while the pair~$(\varnothing, \varnothing)$ is a primitive idempotent of~$\b{k}\f{D}^\star$.
We could have forced connectivity by imposing all short arcs~$(i, i+1, n, \varnothing)$ for~$0 \le i \le n$ in all extended arc ideals of~$\f{J}_n^\star$.
The Hopf algebra~$\b{k}\f{D}^\star$ involves the classes of all lattice congruences of the weak order.
Moreover, the concatenation and selection on extended arc diagrams was chosen to fulfill the following statement.

\begin{proposition}
The permutree Hopf algebra is a Hopf subalgebra of~$\b{k}\f{D}^\star$.
\end{proposition}

\begin{proof}
Any function~$\north, \south : [n] \to \{0,1\}$ naturally correspond to the extended arc ideal of extended arcs~$(a, b, n, S)$ such that~$\north(c) = 0$ for~$c \in S$ and~$\south(c) = 0$ for~$c \in {]a, b[} \ssm S$.
On these particular extended arc ideals, the concatenation and selection corresponds to that defined in Section~\ref{subsec:boundedCrossings}.
The result immediately follows by Proposition~\ref{prop:decorationSubset}.
\end{proof}

In contrast, the reader can check that none of the Hopf algebras of~\cite{LawReading, Giraudo, Pilaud-brickAlgebra} is a Hopf subalgebra of~$\b{k}\f{D}^\star$.

%%%%%%%%%%%%%%%%%%%%%%%%%%%%%%%%%%%%%%

\addtocontents{toc}{\vspace{.3cm}}
\section*{Acknowledgements}

I thank two anonymous referees for helpful comments that improved the presentation, fixed imprecise statements, and suggested relevant bibliographic connections.

%%%%%%%%%%%%%%%%%%%%%%%%%%%%%%%%%%%%%%

\bibliographystyle{alpha}
\bibliography{arcDiagramAlgebra}
\label{sec:biblio}

\end{document}